\documentclass{amsart}

\usepackage{blkarray}
\usepackage{array}
\usepackage{mathrsfs}  
\usepackage{amsmath}
\usepackage{amssymb}
\usepackage[mathscr]{euscript}
\usepackage[all]{xy}
\newcommand{\xyR}[1]{\xydef@\xymatrixrowsep@{#1}}
\newcommand{\xyC}[1]{\xydef@\xymatrixcolsep@{#1}}

\newtheorem{theorem}{Theorem}[section]
\newtheorem{lemma}[theorem]{Lemma}
\newtheorem{corollary}[theorem]{Corollary}
\newtheorem{proposition}[theorem]{Proposition}

\theoremstyle{definition}
\newtheorem{definition}[theorem]{Definition}
\newtheorem{example}[theorem]{Example}
\newtheorem{notation}[theorem]{Notation}
\newtheorem{assumption}[theorem]{Assumption}

\theoremstyle{remark}
\newtheorem{remark}[theorem]{Remark}

\numberwithin{equation}{section}

\begin{document}
\title[Characterising $\Sigma$-pure-injectivity \& endoperfection.]{Characterisations of {$\Sigma$}-pure-injectivity in triangulated categories and applications to endoperfect objects.}
    \author{Raphael Bennett-Tennenhaus.}
\address{  R. Bennett-Tennenhaus\\
Faculty of Mathematics\\
Bielefeld University\\
Universit{\"{a}}t sstra{\ss}e 25\\
33615 Bielefeld\\
 Germany}
\email{raphaelbennetttennenhaus@gmail.com}

\subjclass[2020]{18E45, 18G80 (primary), 03C60 (secondary)}

\begin{abstract}
We provide various ways to characterise $\Sigma$-pure-injective objects in a compactly generated triangulated category. These characterisations mimic analogous well-known results from the model theory of modules. The proof involves two approaches. In the first approach we adapt arguments from the module-theoretic setting. Here the one-sorted language of modules over a fixed ring is replaced with a canonical multi-sorted language, whose sorts are given by compact objects. Throughout we use a variation of the Yoneda embedding, called the resticted Yoneda functor, which associates a multi-sorted structure to each object.  The second approach is to translate statements using this functor. In particular, results about $\Sigma$-pure-injectives in triangulated categories are deduced from results about $\Sigma$-injective objects in Grothendieck categories. Combining the two approaches highlights a connection between sorted pp-definable subgroups and annihilator subobjects of generators in the functor category. Our characterisation motivates the introduction of what we call endoperfect objects, which generalise endofinite objects.
\end{abstract}

\maketitle

\section{Introduction.}
The model theory of modules refers to the specification of model theory to the module-theoretic setting.  Fundamental work, such as that of Baur \cite{Bau1976}, placed focus on certain formulas in the language of modules, known as pp-formulas. In particular, module embeddings which reflect solutions to pp-formulas, so-called pure embeddings, became of particular interest. This served as motivation to study modules which are pure-injective: that is, injective with respect to pure embeddings. 

In famous work of Ziegler \cite{Zie1984}, a topological space was defined whose points are indecomposable pure-injective modules. The introduction of the Ziegler spectrum proved to be a groudbreaking moment in this branch of model-theoretic algebra, and interest in understanding pure-injectivity has since grown. 
Specifically, in work such as that of Huisgen-Zimmerman \cite{Hui1979}, functional results appeared in which pure-injective and so-called $\Sigma$-pure-injective modules were characterised. These characterisations are well documented, for example, by Jensen and Lenzing \cite{JenLen1989}. 

A frequently used tool in these characterisations is the relationship between a module and its image in a certain functor category. To explicate, the functor is given by the tensor product, restricted to the full subcategory of finitely presented modules. For example, a module is pure-injective if and only if the corresponding tensor functor is injective. Subsequently one may convert statements about pure-injective modules into statements about injective objects in Grothendieck categories, and translate problems and solutions back and forth.

For example, Garcia and Dung \cite{GarDun1994} developed the understanding of $\Sigma$-injective objects in Grothendieck categories by building on work of Harada \cite{Har1973}, which generalised a famous characterisation of $\Sigma$-injective modules going back to Faith \cite{Fai1966}. These authors showed that, as above, such developments helped simplify arguments about $\Sigma$-pure-injective modules. 

In this article we attempt to provide, in a utilitarian manner, some analouges to the previously mentioned characterisations. The difference here is that, instead of working in a category of modules, we work in a triangulated category which, in a particular sense, is compactly generated. Never-the-less, the statements we prove and arguments used to prove them are motivated directly from certain module-theoretic counterparts. 

The notion of compactness we refer to comes from work of Neeman \cite{Nee1992}, where the idea was adapted from algebraic topology. Krause \cite{Kra2002} provided the definitions of pure monomorphisms, pure-injective objects and the Ziegler spectrum of a compactly generated triangulated category. Garkusha and Prest \cite{GarPre2005}  subsequently introduced a multi-sorted language for this setting, which mimics the role played by the language of modules. They then gave a correspondence between the pp-formulas in this multi-sorted language and coherent functors. In our main result, Theorem \ref{characterisation}, we use the following notation and assumptions.
\begin{itemize}
\item $\mathcal{T}$ is a compactly generated triangulated category with all small coproducts.
\item $\mathcal{T}^{c}$ is the full subcategory of $\mathcal{T}$ consisting of compact objects, which is assumed to be skeletally small.
\item $\mathbf{Ab}$ is the category of abelian groups.
\item $\mathbf{Mod}\text{-}\mathcal{T}^{c}$ is the category of contravariant additive functors $\mathcal{T}^{c}\to \mathbf{Ab}$.
\item $\mathtt{G}$ is a set of generators of $\mathbf{Mod}\text{-}\mathcal{T}^{c}$ where each $\mathscr{G}\in\mathtt{G}$ is finitely presented.
\item $\mathbf{Y}\colon\mathcal{T}\to\mathbf{Mod}\text{-}\mathcal{T}^{c}$ is the restricted Yoenda functor, which takes an object $M$ to the restriction of the corepresentable $\mathcal{T}(-,M)$.
\end{itemize}
Recall that, by the \textit{Brown representability theorem}, since the category $\mathcal{T}$ has all small coproducts, it has all small products; see Remark \ref{brownproducts}.
\begin{theorem}\label{characterisation}
For any object $M$ of $\mathcal{T}$ the following statements are equivalent.
\begin{enumerate}
\item For any set $\mathtt{I}$ the coproduct $M^{(\mathtt{I})}=\coprod_{\mathtt{I}}M$ is pure-injective.
\item The countable coproduct $M^{(\mathbb{N})}$ is pure-injective.
\item For any generator $\mathscr{G}\in\mathtt{G}$ each ascending chain of $\mathbf{Y}(M)$-annihilator subobjects of $\mathscr{G}$ must eventually stabilise.
\item For any set $\mathtt{I}$ the morphism from $M^{(\mathtt{I})}$ to the product $M^{\mathtt{I}}=\prod_{\mathtt{I}}M$, given by the universal property, is a section.
\item For any object $X$ of $\mathcal{T}^{c}$ each descending chain of pp-definable subgroups of $M$ of sort $X$ must eventually stabilise.
\item $M$ is pure injective, and for any set $\mathtt{I}$, $M^{\mathtt{I}}$ is isomorphic to a coproduct of indecomposable pure-injective objects with local endomorphism rings.
\end{enumerate}
\end{theorem}
The proof of Theorem \ref{characterisation} is at the end of the article. The equivalences of (1), (2), (3) and (4) in Theorem \ref{characterisation} follow by directly combining work of Garcia and Dung \cite{GarDun1994} and work of Krause \cite{Kra2002}. The equivalence of (5) and (6) with the previous conditions is more involved. For (5) we adapt ideas going back to Faith \cite{Fai1966}, whilst applying results due to Harada \cite{Har1973} and Garcia and Dung \cite{GarDun1994}. For (6) we adapt arguments of Huisgen-Zimmerman \cite{Hui1979}.

The article is organised as follows. In \S\ref{1} we recall some prerequisite terminology from multi-sorted model theory. In \S\ref{2} we specify to compactly generated triangulated categories by recalling the canonical multi-sorted language of Garkusha and Prest \cite{GarPre2005}. In \S\ref{3} we gather results about products and coproducts of pp-definable subgroups in this context, following ideas of Huisgen-Zimmerman \cite{Hui1979}. In \S\ref{4} we highlight a connection between annihilator subobjects of finitely generated functors and pp-definable subgroups, where the presentation of the functor determines the sort of the subgroup. In \S\ref{4-1} we begin combining the results developed in the previous sections with results of Krause \cite{Kra2002}. In \S\ref{5} we complete the proof of Theorem \ref{characterisation}. In \S\ref{applicationstoendoperfect} we introduce \emph{endoperfect} objects, and see applications of Theorem \ref{characterisation}.

\section{Multi-sorted languages, structures and homomorphisms.}\label{1}
There are various module-theoretic characterisations of purity in terms of \textit{positive}-\textit{primitive} formulas in the underlying \textit{one}-\textit{sorted} language of modules over a ring. Similarly, purity in compactly generated triangulated categories may be discussed in terms of formulas in a multi-sorted language. Although Definitions \ref{sorts}, \ref{sorts2-1}, \ref{sorts2}, \ref{sorts3}  and \ref{sorts4} are well-known, we recall them for completeness. We closely follow \cite[\S2, \S 7]{DGN2016} for consistency.

\begin{definition}\label{sorts}\cite[Definition 34]{DGN2016} For a non-empty set $\mathtt{S}$, an $\mathtt{S}$-\textit{sorted predicate language} $\mathfrak{L}$ is a tuple $\langle\mathrm{pred}_{\mathtt{S}},\,\mathrm{func}_{\mathtt{S}},\,\mathrm{ar}_{\mathtt{S}},\,\mathrm{sort}_{\mathtt{S}}\rangle$ where: 
\begin{enumerate}
\item each $s\in\mathtt{S}$ is called a \textit{sort}; 
\item the symbol $\mathrm{pred}_{\mathtt{S}}$ denotes a non-empty set of \textit{sorted predicate symbols}; 
\item the symbol $\mathrm{func}_{\mathtt{S}}$ denotes a set of \textit{sorted function symbols}; 
\item the symbol $\mathrm{ar}_{\mathtt{S}}$ denotes an \textit{arity} function $\mathrm{pred}_{\mathtt{S}}\sqcup\mathrm{func}_{\mathtt{S}}\to\mathbb{N}=\{0,1,2,\dots\}$;
\item the symbol $\mathrm{sort}_{\mathtt{S}}$ denotes a \textit{sort} function, taking any $n$-ary $R\in\mathrm{pred}_{\mathtt{S}}$ (respectively $F\in\mathrm{func}_{\mathtt{S}}$) to a sequence in $\mathtt{S}$ of length $n$ (respectively $n+1$).
\end{enumerate}
When $n>0$ in condition (5) we often write $\mathrm{sort}_{\mathtt{S}}(R)=(s_{1},\dots,s_{n})$ (respectively $\mathrm{sort}_{\mathtt{S}}(F)=(s_{1},\dots,s_{n},s)$). Note that functions $F$ with $\mathrm{ar}_{\mathtt{S}}(F)=0$ have a sort.

For each sort $s$ we introduce a countable set $\mathcal{V}_{s}$ of \textit{variables of sort} $s$. The \textit{terms} of $\mathfrak{L}$ each have their own sort, and are defined inductively by stipulating: any variable  $x$ of sort $s$ will be considered a term of sort $s$; and for any $F\in\mathrm{func}_{\mathtt{S}}$ with $\mathrm{sort}_{\mathtt{S}}(F)=(s_{1},\dots,s_{n},s)$ and for any terms $t_{1},\dots,t_{n}$ of sort $s_{1},\dots,s_{n}$ respectively, $F(t_{1},\dots,t_{n})$ is considered a term of sort $s$. Note that \textit{constant symbols}, given by functions $F$ with $\mathrm{ar}_{\mathtt{S}}(F)=0$, are also terms.

The \textit{atomic formulas} with which $\mathfrak{L}$ is equipped are built from the equality $t=_{s}t'$ between terms $t,t'$ of common sort $s$, together with the formulas $R(t_{1},\dots,t_{n})$ where $R \in\mathrm{pred}_{\mathtt{S}}$, $\mathrm{sort}_{\mathtt{S}}(R)=(s_{1},\dots,s_{n})$ and where each $t_{i}$ is a term of sort $s_{i}$. First-order formulas $\varphi$ in $\mathfrak{L}$ are built from: the variables of each sort; the atomic formulas; binary connectives $\wedge$, $\vee$, and $\implies$; negation $\neg$; and the quantifiers $\forall$ and $\exists$.  
 
A \textit{positive}-\textit{primitive} or \textit{pp} formula $\varphi(x_{1},\dots,x_{n})$ (with $x_{i}$ free) has the form
\[\begin{array}{c}
\exists\, w_{n+1},\dots,w_{m}\colon\bigwedge_{j=1}^{k}\psi_{j}(x_{1},\dots,x_{n},w_{n+1},\dots,w_{m}),
\end{array}
\]
where each $\psi_{j}$ is an atomic formula (see, for example, \cite[p.50]{Hod1997}). 
\end{definition}
One may build a \textit{theory} for a multi-sorted language $\mathfrak{L}$ by specficying a set of axioms. For our purposes these axioms are those charaterising objects and morphisms in a fixed category. We explain this idea by means of examples.
\begin{example}\label{onesorted}\cite[\S 6]{JenLen1989} Let $A$ be a unital ring. We recall how the language $\mathfrak{L}_{A}$ of $A$-modules may be considered as a predicate language in the sense of Definition \ref{sorts}. In this case there is only one sort, which we ignore, and which uniquely determines the function $\mathrm{sort}_{A}$. Let $\mathrm{pred}_{A}=\{0\}$. Let $\mathrm{func}_{A}=\{+\}\cup\{a\times-\mid a\in A\}$ where $+$ is binary and $a\times-$ is unary. 

In Definition \ref{sorts2-1} the notion of a \textit{structure} is recalled. For the language $\mathfrak{L}_{A}$ we have that this notion, together with certain axioms, recovers the defining properties of left $A$-modules. Later we consider homomorphisms. 
\end{example}
\begin{definition}\label{sorts2-1}
\cite[Definition 35]{DGN2016} Fix a set $\mathtt{S}\neq\emptyset$  and an $\mathtt{S}$-sorted predicate language $\mathfrak{L}$. An  $\mathfrak{L}$-\textit{structure} is a tuple \[\mathsf{M}=\langle \mathtt{S}(\mathsf{M}), (R(\mathsf{M})\mid R \in\mathrm{pred}_{\mathtt{S}}), (F(\mathsf{M})\mid F \in\mathrm{func}_{\mathtt{S}})\rangle\] such that: 
\begin{enumerate}
\item the symbol $\mathtt{S}(\mathsf{M})$ denotes a family of sets $\{s(\mathsf{M})\mid s\in\mathtt{S}\}$; 
\item if $\mathrm{sort}_{\mathtt{S}}(R)=(s_{1},\dots,s_{n})$ then $R(\mathsf{M})$ is a subset of $s_{1}(\mathsf{M})\times\dots\times s_{n}(\mathsf{M})$; and 
\item if $\mathrm{sort}_{\mathtt{S}}(F)=(s_{1},\dots,s_{n},s)$ then $F(\mathsf{M})$ is a map $s_{1}(\mathsf{M})\times\dots\times s_{n}(\mathsf{M})\to s(\mathsf{M})$.  
\end{enumerate}
Denoting the cardinality of any set $\mathtt{X}$ by $\vert \mathtt{X}\vert$, let $\vert\mathfrak{L}\vert=\vert \mathrm{pred}_{\mathtt{S}}\sqcup\mathrm{func}_{\mathtt{S}}\vert$ and, for $\mathsf{M}$ as above, let $\vert\mathsf{M}\vert$ be the sum  of the cardinalities $\vert s(\mathsf{M})\vert$ as $s$ runs through $\mathtt{S}$.
\end{definition}
The so-called \textit{one}-\textit{sorted} language from Example \ref{onesorted} is trivial in the sense that there is only one possibility for the sort function. In this way, Example \ref{twosorted} is a non-trivial example of the \textit{multi}-\textit{sorted} languages we recalled in Definition \ref{sorts}.
\begin{example}\label{twosorted}\cite[\S 9]{JenLen1989} Here we recall an example of an $\{\mathtt{r},\mathtt{m}\}$-sorted predicate language which is in contrast to Example \ref{onesorted}. The predicates in this language will be the unary symbols $0_{\mathtt{r}}$ and $1_{\mathtt{r}}$ of sort $\mathtt{r}$, and $0_{\mathtt{m}}$ of sort $\mathtt{m}$. The functions in this language will be the ternary symbols $+$ and $\times$ where $\mathrm{sort}_{\mathtt{r},\mathtt{m}}(+)=(\mathtt{m},\mathtt{m},\mathtt{m})$ and $\mathrm{sort}_{\mathtt{r},\mathtt{m}}(\times)=(\mathtt{r},\mathtt{m},\mathtt{m})$. After specifying the appropriate axioms, structures $_{\mathsf{A}}\mathsf{M}$ are tuples $(A,M)$ where $A$ is a unital ring and $M$ is a left $A$-module. 

In this way one interprets the symbols $0_{\mathtt{r}}$ and $1_{\mathtt{r}}$ as the additive and multiplicative identities in $A$. Similarly the symbol $0_{\mathtt{m}}$ is interpreted as the additive identity in $M$. 
In Definition \ref{sorts2} the notion of a \textit{homomorphism} between structures is recalled. In this sense, a homomorphism $(A,M)\to(B,N)$ is given by a pair $(f,l)$ where $f\colon A\to B$ is a homomorphism of rings and $l\colon M\to N$ is a homomorphism of left $A$-modules with the action of $A$ on $N$ given by $f$.
\end{example}
\begin{definition}\label{sorts2}
\cite[Definition 3]{DGN2016} Fix a non-empty set $\mathtt{S}$, an $\mathtt{S}$-sorted predicate language $\mathfrak{L}$ and  $\mathfrak{L}$-structures $\mathsf{L}$ and $\mathsf{M}$. By an $\mathfrak{L}$-\textit{homomorphism} $\mathsf{h}\colon\mathsf{L}\to\mathsf{M}$ we mean a family $\{\mathsf{h}_{s}\mid s\in\mathtt{S}\}$ of functions $\mathsf{h}_{s}\colon s(\mathsf{L})\to s(\mathsf{M})$ such that: 
\begin{enumerate}
\item if $\mathrm{sort}_{\mathtt{S}}(R)=(s_{1},\dots,s_{n})$ then $R(\mathsf{M})$ is the set of $(\mathsf{h}_{s_{1}}(a_{1}),\dots,\mathsf{h}_{s_{n}}(a_{n}))$ such that $(a_{1},\dots,a_{n})\in R(\mathsf{L})$;
\item and if $\mathrm{sort}_{\mathtt{S}}(F)=(s_{1},\dots,s_{n},s)$ then  for all $(a_{1},\dots,a_{n})\in s_{1}(\mathsf{L})\times \dots \times s_{n}(\mathsf{L})$ we have $\mathsf{h}_{s}(F(\mathsf{L})(a_{1},\dots,a_{n}))=F(\mathsf{M})(\mathsf{h}_{s_{1}}(a_{1}),\dots,\mathsf{h}_{s_{n}}(a_{n}))$.
\end{enumerate}
\end{definition}
Note that, in the notation of Definition \ref{sorts2}, \cite[Theorem 17]{DGN2016} says that a collection of functions $\mathsf{h}_{s}:s(\mathsf{L})\to s(\mathsf{M})$  defines an $\mathfrak{L}$-homomorphism if and only if, whenever $\varphi(x_{1},\dots,x_{n})$ is an atomic formula with $\mathrm{sort}_{\mathtt{S}}(x_{i})=s_{i}$ and $(a_{1},\dots,a_{n})$ lies in $s_{1}(\mathsf{L})\times \dots \times s_{n}(\mathsf{L})$, we have that 
\[
\mathsf{L}\models\varphi(a_{1},\dots,a_{n})\implies\mathsf{M}\models\varphi(\mathsf{h}_{s_{1}}(a_{1}),\dots,\mathsf{h}_{s_{n}}(a_{n})).
\]
We are now ready to recall the idea of purity coming from model theory.
\begin{definition}\label{sorts3}\cite[Definition 36]{DGN2016}
Fix a set $\mathtt{S}\neq\emptyset$  and an $\mathtt{S}$-sorted predicate language $\mathfrak{L}$. By an $\mathfrak{L}$-\textit{embedding} we mean an $\mathfrak{L}$-homomorphism $\mathsf{h}\colon\mathsf{L}\to\mathsf{M}$ such that:
\begin{enumerate}
\item if $\varphi(x_{1},\dots,x_{n})$ is an atomic formula with $\mathrm{sort}_{\mathtt{S}}(x_{i})=s_{i}$, then for all $(a_{1},\dots,a_{n})\in s_{1}(\mathsf{L})\times \dots \times s_{n}(\mathsf{L})$, $\mathsf{L}\models\varphi(a_{1},\dots,a_{n})$ if and only if $\mathsf{M}\models\varphi(\mathsf{h}_{s_{1}}(a_{1}),\dots,\mathsf{h}_{s_{n}}(a_{n}))$.
\end{enumerate}
By an $\mathfrak{L}$-\textit{pure embedding} we mean an $\mathfrak{L}$-homomorphism $\mathsf{h}\colon\mathsf{L}\to\mathsf{M}$ such that:
\begin{enumerate}\setcounter{enumi}{1}
\item if $\varphi(x_{1},\dots,x_{n})$ is a pp formula with $\mathrm{sort}_{\mathtt{S}}(x_{i})=s_{i}$, then for all $(a_{1},\dots,a_{n})\in s_{1}(\mathsf{L})\times \dots \times s_{n}(\mathsf{L})$, if $\mathsf{M}\models\varphi(\mathsf{h}_{s_{1}}(a_{1}),\dots,\mathsf{h}_{s_{n}}(a_{n}))$ then $\mathsf{L}\models\varphi(a_{1},\dots,a_{n})$.
\end{enumerate}
\end{definition}
Note that $\mathfrak{L}$-pure embeddings are $\mathfrak{L}$-embeddings. Note also that the statement of Definition \ref{sorts2}(2) is the contrapositive of the definition in \cite[p.50]{Hod1997}, so in this sense, over a ring $A$ and in the notation from Example \ref{onesorted}, an injective left $A$-module homomorphism is pure if and only if it is an $\mathfrak{L}_{A}$-pure embedding.
\begin{definition}\label{sorts4}
Fix a non-empty set $\mathtt{S}$, an $\mathtt{S}$-sorted predicate language $\mathfrak{L}$ and $\mathfrak{L}$-structures $\mathsf{L}$ and $\mathsf{M}$. We say $\mathsf{L}$ is an $\mathfrak{L}$-\textit{substructure} of $\mathsf{M}$ if $s(\mathsf{L})\subseteq s(\mathsf{M})$ for each $s\in\mathtt{S}$ and, labelling these inclusions $\mathsf{i}_{s}$, the family $\{\mathsf{i}_{s}\mid s\in\mathtt{S}\}$ defines an $\mathfrak{L}$-homomorphism $\mathsf{i}\colon\mathsf{L}\to\mathsf{M}$. If, additionally, $\mathsf{i}\colon\mathsf{L}\to\mathsf{M}$ is an $\mathfrak{L}$-pure embedding, we say $\mathsf{L}$ is an $\mathfrak{L}$-\textit{pure substructure} of  $\mathsf{M}$.
\end{definition}
\section{Purity in the canonical language of a triangulated category.}\label{2}
We now specify the setting of multi-sorted model theory outlined in \S\ref{1}. Throughout the sequel we consider a fixed \textit{compactly generated} triangulated category; see Assumption \ref{ass22}. Before recalling Definition \ref{compactob} we fix some notation.
\begin{notation}\label{categnotation}
Let $\mathcal{A}$ be an additive category. Denote the hom-sets $\mathcal{A}(X,Y)$ and the identity maps $1_{X}$. For any set $\mathtt{I}$ and any collection $\mathrm{B}=\{B_{i}\mid i\in \mathtt{I}\}$ of objects in $\mathcal{A}$, if the categorical product $\prod_{i}B_{i}$ exists in $\mathcal{A}$, we write $p_{j,\mathrm{B}}\colon\prod_{i}B_{i}\to B_{j}$ for the natural morphisms equipping it, in which case the universal property gives unique morphisms $v_{j,\mathrm{B}}\colon B_{j}\to \prod_{i}B_{i}$ such that $p_{j,\mathrm{B}}v_{j,\mathrm{B}}$ is the identity $1_{j}$ on $B_{j}$ for each $j$. Similarly $u_{j,\mathrm{B}}\colon B_{j}\to\coprod_{i}B_{i}$ will denote the morphisms equipping the coproduct $\coprod_{i}B_{i}$ if it exists, in which case there exist unique morphisms $q_{j,\mathrm{B}}\colon\coprod_{i}B_{i}\to B_{j}$ such that $q_{j,\mathrm{B}}u_{j,\mathrm{B}}=1_{j}$ for each $j$.

Fix an object $A$ in $\mathcal{A}$ and consider the covariant functor $\mathcal{A}(A,-)$. Note that both the product and coproduct of the collection $\mathcal{A}(A,\mathrm{B})=\{\mathcal{A}(A,B_{i})\mid i\in \mathtt{I}\}$ exist in the category $\mathbf{Ab}$ of abelian groups. We identify $\coprod_{i\in \mathtt{I}}\mathcal{A}(A,B_{i})$ with the subgroup of $\prod_{i\in \mathtt{I}}\mathcal{A}(A,B_{i})$ consisting of tuples $(g_{i}\mid i\in \mathtt{I})$ such that $g_{i}=0$ for all but finitely many $i\in \mathtt{I}$. 

Consequently, if $\prod_{i}B_{i}$ exists in $\mathcal{A}$ then map $\lambda_{A,\mathrm{B}}\colon\mathcal{A}(A,\prod_{i} B_{i})\to \prod_{i} \mathcal{A}(A,B_{i})$ from the universal property is given by $f\mapsto(p_{i,\mathrm{B}}f\mid i\in \mathtt{I})$ for each $f\in\mathcal{A}(A,\prod_{i} B_{i})$. Similarly if $\coprod_{i}B_{i}$ exists in $\mathcal{A}$ then map $\gamma _{A,\mathrm{B}}\colon\coprod _{i}\mathcal{A}(A, B_{i})\to \mathcal{A}(A,\coprod _{i}B_{i})$ from the universal property is given by $\gamma_{A,\mathrm{B}}(g_{i}\mid i\in \mathtt{I})=\sum_{i}u_{i,\mathrm{B}}g_{i}$. In general each of the morphisms $\lambda_{A,\mathrm{B}}$ are isomorphisms.

\end{notation}
\begin{definition}\label{compactob}\cite[Definition 1.1]{Nee1992} Let $\mathcal{T}$ be a triangulated category with suspension functor $\Sigma$, and assume all small coproducts in $\mathcal{T}$ exist. An object $X$ of $\mathcal{T}$ is said to be \textit{compact} if, for any set $\mathtt{I}$ and collection $\mathrm{M}=\{M_{i}\mid i\in \mathtt{I}\}$ of objects in $\mathcal{T}$ the morphism $\gamma_{X,\mathrm{M}}$ is an isomorphism. Let $\mathcal{T}^{c}$ be the full triangulated subcategory of $\mathcal{T}$ consisting of compact objects.

Given a set $\mathcal{G}$ of compact objects in $\mathcal{T}$ we say that $\mathcal{T}$ is \textit{compactly generated by} $\mathcal{G}$ if there  are no non-zero objects $M$ in $\mathcal{T}$ satisfying $\mathcal{T}(X,M)=0$ for all $X\in\mathcal{G}$ (or, said another way, any non-zero object $M$ gives rise to a non-zero morphism $X\to M$ for some $X\in\mathcal{G}$). If $\mathcal{T}$ is compactly generated by $\mathcal{G}$ we call $\mathcal{G}$ a \textit{generating set} provided $\Sigma X\in \mathcal{G}$ for all $X\in\mathcal{G}$. 
\end{definition}
\begin{assumption}\label{ass22}In the remainder of \S\ref{2} fix a triangulated category $\mathcal{T}$ with suspension functor $\Sigma$, and we assume:
\begin{itemize}
\item that  $\mathcal{T}$ has all small coproducts;
\item that $\mathcal{T}$ is compactly generated by a generating set $\mathcal{G}$; and
\item that the subcategory $\mathcal{T}^{c}$ of compact objects is skeletally small. 
\end{itemize}
\end{assumption}
Definition \ref{canonlang} and Remark \ref{canonlangrem} closley follow \cite[\S 3]{GarPre2005}, in which a multi-sorted language associated to the category $\mathcal{T}$ is introduced.

\begin{definition}\label{choiceofskel} In what follows let $\mathcal{S}$ denote a fixed set of objects in $\mathcal{T}^{c}$ given by choosing exactly one representative of each isomorphism class. Such a set $\mathcal{S}$ exists because we are assuming that $\mathcal{T}^{c}$ is skeletally small. 

\label{canonlang} \cite[\S 3]{GarPre2005} The \textit{canonical language} $\mathfrak{L}^{\mathcal{T}}$ of $\mathcal{T}$ is given by an $\mathcal{S}$-sorted predicate language $\langle\mathrm{pred}_{\mathcal{S}},\,\mathrm{func}_{\mathcal{S}},\,\mathrm{ar}_{\mathcal{S}},\,\mathrm{sort}_{\mathcal{S}}\rangle$, defined as follows. The set $\mathrm{pred}_{\mathcal{S}}$ consists of a symbol $\mathsf{0}_{G}$ with $\mathrm{sort}_{\mathcal{S}}(\mathsf{0}_{G})=G$ for each $G\in\mathcal{S}$. The set $\mathrm{func}_{\mathcal{S}}$ consists of: a ternary symbol $+_{G}$ with $\mathrm{sort}_{\mathcal{S}}(+_{G})=(G,G,G)$ for each $G\in\mathcal{S}$; and a unary operation $-\circ a$ with $\mathrm{sort}_{\mathcal{S}}(-\circ a)=(H,G)$ for each morphism $a:G\to H$ with $G,H\in\mathcal{S}$. Variables of sort $G\in\mathcal{S}$ will be denoted $v_{G}$.
\end{definition}
\begin{notation}
Suppose $\mathcal{A}$ is any additive category. We write $\mathcal{A}\text{-}\mathbf{Mod}$ (respectively $\mathbf{Mod}\text{-}\mathcal{A}$) for the category of additive covariant (respectively contravariant) functors $\mathcal{A}\to \mathbf{Ab}$ where $\mathbf{Ab}$ is the category of abelian groups.

For any object $M$ of $\mathcal{T}$ we let $\mathcal{T}(-,M)\vert$ denote the object of $\mathbf{Mod}\text{-}\mathcal{T}^{c}$ defined by restriction of $\mathcal{T}(-,M)$ to compact objects. We write \[
\mathbf{Y}\colon\mathbf{Mod}\text{-}\mathcal{T}^{c}\to\mathbf{Ab}\]
 to denote the \textit{restricted Yoneda functor}. That is, $\mathbf{Y}$ takes an object $M$ to $\mathbf{Y}(M)=\mathcal{T}(-,M)\vert$, and takes a morphism $h:L\to M$ to the natural transformation $\mathbf{Y}(h)\colon\mathcal{T}(-,L)\vert\to\mathcal{T}(-,M)\vert$ given by defining, for each compact object $X$, the map $\mathbf{Y}(h)_{X}\colon\mathcal{T}(X,L)\to\mathcal{T}(X,M)$ by $g\mapsto hg$.
\end{notation}
\begin{remark}\label{canonlangrem}\cite[\S 3]{GarPre2005} Consider the theory given from the set of axioms expressing the positive atomic diagram of the objects in $\mathcal{T}^{c}$ including the specification that all functions are additive. In this way, the category of models for the above theory is equivalent to the category $\mathbf{Mod}\text{-}\mathcal{T}^{c}$ where objects $M$ of $\mathcal{T}$ are regarded as structures $\mathsf{M}$ for this language. 

That is, in the notation of Definition \ref{sorts2-1}, we let $G(\mathsf{M})=\mathcal{T}(G,M)$, we interpret the predicate symbol $0_{G}$ as the identitly element of $G(\mathsf{M})$, we interpret $+_{G}$ as the additive group operation on $G(\mathsf{M})$, and we interpret $-\circ a$ as the map $G(\mathsf{M})\to H(\mathsf{M})$ given by $f\mapsto fa$. 
\end{remark}
\begin{lemma}\label{elempure}Let $L$ and $M$ be objects in $\mathcal{T}$ with corresponding $\mathfrak{L}^{\mathcal{T}}$-structures $\mathsf{L}$ and $\mathsf{M}$. Then the choice of $\mathcal{S}$ made in Definition \ref{choiceofskel} defines a bijection between morphisms $\mathbf{Y}(L)\to\mathbf{Y}(M)$ (that is, natural transformations) in $\mathbf{Mod}\text{-}\mathcal{T}^{c}$ and $\mathfrak{L}^{\mathcal{T}}$-homomorphisms $\mathsf{L}\to\mathsf{M}$.
\end{lemma}
\begin{proof}Any object $X$ of $\mathcal{T}^{c}$ lies in the same isoclass as some unique $c(X)\in\mathcal{S}$, in which case we choose an isomorphism $\phi_{X}\colon c(X)\to X$. In this way, any object $N$ defines an isomorphism $-\circ\phi_{X,N}\colon\mathcal{T}(X,N)\to\mathcal{T}(c(X),N)$ by precomposition with $\phi_{X}$.  Recall that here the $\mathfrak{L}^{\mathcal{T}}$-structure $\mathsf{N}$ is defined by setting $G(\mathsf{N})=\mathcal{T}(G,N)$ for each sort $G\in\mathcal{S}$. In case $X\in\mathcal{S}$ we assume, without loss of generality, that $\phi_{X}=1_{X}$. 
Define the required bijection as follows. Fix an $\mathfrak{L}^{\mathcal{T}}$-homomorphism $\mathsf{h}\colon\mathsf{L}\to\mathsf{M}$. For any object $X$ of $\mathcal{T}^{c}$ define the function $\mathscr{H}(\mathsf{h})_{X}\colon\mathcal{T}(X,L)\to\mathcal{T}(X,M)$ by $l\mapsto (\mathsf{h}_{c(X)}(l\phi_{X}))\phi^{-1}_{X}$. Converlsey, fixing a natural transformation $\mathscr{H}\colon\mathbf{Y}(L)\to\mathbf{Y}(M)$, let $\mathsf{h}(\mathscr{H})_{G}=\mathscr{H}_{G}$ for each $G\in\mathcal{S}$. 

It suffices to explain why these assignments swap between morphisms in $\mathbf{Mod}\text{-}\mathcal{T}^{c}$ and $\mathfrak{L}^{\mathcal{T}}$-homomorphisms. To do so, we explain why the compatability conditions which define these morphisms are in correspondence. To this end, note firstly that the preservation of (the predicate symbol $0_{G}$ and the function symbols $+_{G}$) is equivalent to saying that each function $\mathscr{H}_{X}$ is a homomorphism of abelian groups. Letting $b\colon X\to Y$ be a morphism in $\mathcal{T}^{c}$ and $a=\phi_{Y}^{-1}b\phi_{X}$, for any object $N$ of $\mathcal{T}$ the sorted function symbol $-\circ a$ is interpreted in $\mathsf{N}$ by the equation $(-\circ a)(\mathsf{N})=-\circ(\phi_{N,Y}^{-1})b\phi_{N,X}$. Thus, by construction, saying that the function symbols $-\circ a$ are preserved is equivalent to saying that the collection of $\mathscr{H}_{X}$ (for $X$ compact) defines a natural transformation.
\end{proof}
In what follows we discuss the notion of purity in the context of compactly generated triangulated categories. 
\begin{definition}\cite[Definition 1.1]{Kra2002} We say $h\colon L\to M$ in $\mathcal{T}$ is a \textit{pure monomorphism} if $\mathbf{Y}(h)_{X}\colon\mathcal{T}(X,L)\to\mathcal{T}(X,M)$ is injective for each compact object $X$.
\end{definition}
Now we may begin to build results which mimic well-known ideas from the model theory of modules. To consistently compare and contrast our work with the module-theoretic setting, we use a book of Jensen and Lenzing \cite{JenLen1989}. In this spirit, Lemma \ref{elempure2} is analogous to \cite[Theorem 6.4(i,ii)]{JenLen1989}, and  Lemma \ref{ppprops} is analogous to \cite[Proposition 6.6]{JenLen1989}. Similar analogies are found throughout the sequel.
\begin{lemma}\label{elempure2}A morphism $h\colon L\to M$ is a pure monomorphism if and only if the image $\mathsf{h}\colon \mathsf{L}\to\mathsf{M}$ under the bijection in Lemma \ref{elempure} is an $\mathfrak{L}^{\mathcal{T}}$-pure embedding.
 \end{lemma}
 \begin{proof}By \cite[Proposition 3.1]{GarPre2005} any pp-formula $\varphi(v_{G})$ is equivalent to a divisibility formula $\exists u_{H} \colon v_{G}=u_{H}a$ where $a\colon G\to H$ is morphism and $G,H\in\mathcal{S}$. By Definition \ref{sorts3}, $\mathsf{h}$ is an $\mathfrak{L}^\mathcal{T}$-pure embedding if and only if, for any morphism $a\colon G\to H$ with $G,H\in\mathcal{S}$ and any pair $(f,g)\in G(\mathsf{L})\times H(\mathsf{L})$, if $hg=h fa$ then $g=fa$. Since any compact object is isomorphic to an object in $\mathcal{S}$, this is equivalent to the condition which says that, for each compact object $X$, the morphism $\mathcal{T}(X,L)\to\mathcal{T}(X,M)$ given by $g\mapsto h g$ is injective. 
\end{proof} 
\section{Products, coproducts, coherent functors and pp-formulas.}\label{3}
Recall, from Definition \ref{sorts}, that pp-formulas in $\mathfrak{L}^{\mathcal{T}}$ are those lying in the closure of the set of equations under conjunction and existential quantification.
\begin{definition}\cite[\S 2]{GarPre2005} Given $G\in\mathcal{S}$ and an object $M$ of $\mathcal{T}$ with $\mathfrak{L}^{\mathcal{T}}$-structure $\mathsf{M}$, a \textit{pp}-\textit{definable subgroup of} $M$ \textit{of sort} $G$ is the set $\varphi(M)=\{f\in G(\mathsf{M})\mid \mathsf{M}\models\varphi(f)\}$ of solutions (in $\mathsf{M}$) to some pp-formula $\varphi(v_{G})$ in one free variable of sort $G\in\mathcal{S}$. 

For any morphism $b\colon X\to Y$ in $\mathcal{T}^{c}$ and any object $M$ in $\mathcal{T}$ recall the map $\mathbf{Y}(b)\colon \mathcal{T}(Y, M)\to \mathcal{T}(X, M)$ is defined by precomposition. In this case let
\[
Mb=\mathrm{im}(\mathbf{Y}(b))=\{fb\in \mathcal{T}(X, M)\mid f\in\mathcal{T}(Y, M)\}.
\]
If $G,H\in\mathcal{S}$ and $\phi_{X}\colon G\to X$ and $\phi_{Y}\colon H\to Y$ are isomorphisms in $\mathcal{T}$ (as in the proof of Lemma \ref{elempure}), then $fb\mapsto f\phi^{-1}_{Y}b\phi_{X}$ defines a isomorphism $Mb\to\varphi(M)$ in $\mathbf{Ab}$ where $\varphi(v_{G})$ is the pp-formula $(\exists u_{H} \colon v_{G}=u_{H}a)$ where $a=\phi^{-1}_{Y}b\phi_{X}$. 
\end{definition}
\begin{lemma}\label{ppprops}Let $G\in\mathcal{S}$ and let $\varphi(v_{G})$ be a pp-formula in one free variable of sort $G$. If $h\colon L\to M$ is a pure monomorphism then $\varphi(L)=\{g\in\mathcal{T}(G,L)\mid hg\in\varphi(M)\}$. 
\end{lemma}
\begin{proof}
The claim follows from Lemma \ref{elempure2}, together with Definition \ref{sorts3}, which we recall defines $\mathfrak{L}^{\mathcal{T}}$-pure embeddings as those which preserve solutions to the negations of all pp-formulas.
\end{proof}
We continue, slightly abusing terminology, by reffering to any set of the form $Mb$ (for some $b\in\mathcal{T}(X,Y)$) as a \textit{pp}-\textit{definable subgroup of} $M$ \textit{of sort} $X$. Lemma \ref{endodefinable} is an analgoue of the corresponding result for module categories; see for example \cite[Corollary 2.2(i)]{Pre1988}. Note that we state it here for convenience, but that it is only used in \S\ref{applicationstoendoperfect}, and not in the proof of Theorem \ref{characterisation}. 
\begin{lemma}\label{endodefinable}
For any object $M$ in $\mathcal{T}$ and any morphism $b\in\mathcal{T}(X,Y)$ the pp-definable subgroup $Mb$ is a left $\mathrm{End}_{\mathcal{T}}(M)$-submodule of $\mathcal{T}(X, M)$. 
\end{lemma}
\begin{proof}By the associativity of the composition of morphisms in $\mathcal{T}$ the set $Mb$ is closed under postcomposition with endomorphisms of $M$.
\end{proof}
Recall that a covariant functor $\mathscr{F}:\mathcal{T}\to\mathbf{Ab}$ is \textit{coherent} if there is an exact sequence in $\mathcal{T}\text{-}\mathbf{Mod}$ of the form $\mathcal{T}(A,-)\to\mathcal{T}(B,-)\to\mathscr{F}\to0$. If $t\colon M\to N$ and $a:G\to H$ are morphisms in $\mathcal{T}$ with $G,H\in\mathcal{S}$, then $tv\in Na$ for any $v\in Ma$. So, for the pp-formula $\varphi(v_{G})=(\exists u_{H} \colon v_{G}=u_{H}a)$ in $\mathfrak{L}^{\mathcal{T}}$ the assignment of objects $M\mapsto \varphi(M)$ defines a functor. 

Furthermore, by \cite[Lemma 4.3]{GarPre2005} these functors are coherent, and any such coherent functor arises this way. We now recall that the categories we are considering have all small products.
\begin{remark}\label{brownproducts}
As a result of Assumption \ref{ass22}, by the \textit{Brown representability theorem} we have that $\mathcal{T}$ has all small products. See \cite[Lemma 1.5]{Kra2000} for details.
\end{remark}
Lemma \ref{ppprops2} is analogous to \cite[Proposition 6.7(i,ii)]{JenLen1989}. Recall Notation \ref{categnotation}.
\begin{lemma}\label{ppprops2}Let $G\in\mathcal{S}$ and let $\varphi(v_{G})$ be a pp-formula in one free variable of sort $G$. For any set $I$ and any collection $\mathrm{M}=\{M_{i}\mid i\in \mathtt{I}\}$ of objects in $\mathcal{T}$ the restrictions of $\gamma_{G,\mathrm{M}}$ and $\lambda_{G,\mathrm{M}}$ define isomorphisms of abelian groups
\[
\begin{array}{cc}
\coprod_{i}\varphi(M_{i})\to \varphi(\coprod_{i}M_{i}), & \varphi(\prod_{i}M_{i})\to \prod_{i}\varphi(M_{i}).
\end{array}\]
\end{lemma}
\begin{proof}By the existence of small products and coproducts in $\mathcal{T}$ and the functorality of $\varphi$, the universal properties give morphisms $\delta\colon\coprod_{i}\varphi(M_{i})\to \varphi(\coprod_{i}M_{i})$ and $\mu\colon \varphi(\prod_{i}M_{i})\to \prod_{i}\varphi(M_{i})$.
By \cite[Lemma 4.3]{GarPre2005} the functor $\varphi$ is coherent, so by the equivalence of statements (1) and (3) from \cite[Theorem A]{Kra2002} the morphisms $\delta$ and $\mu$ are isomorphisms. It is straightforward to check that $\delta$ and $\mu$ are the respectively restrictions of $\gamma_{G,\mathrm{M}}$ and $\lambda_{G,\mathrm{M}}$ from Notation \ref{categnotation}.
\end{proof}
We now adapt some technical results from work of Huisgen-Zimmerman \cite{Hui1979}, in which a (now well-known) charaterisation of $\Sigma$-pure-injective modules was given. Our adaptations, namely Lemmas \ref{zimmtech2} and \ref{zimmtech3}, are used in the sequel.
\begin{notation}\label{techynotation}Fix collections $\mathrm{M}=\{M_{i}\mid i\in \mathbb{N}\}$ and $\mathrm{L}=\{L_{j}\mid j\in \mathtt{J}\}$ of objects in $\mathcal{T}$ and let $M=L$ where $M=\prod_{i}M_{i}$ and $L=\coprod_{j}L_{j}$ exist.  By Notation \ref{categnotation},  $p_{i,\mathrm{M}}\colon M\to M_{i}$ and $u_{j,\mathrm{L}}\colon L_{j}\to L$ denote the morphisms equipping the product and coproduct, and $v_{i,\mathrm{M}}\colon M_{i}\to M$ and  $q_{j,\mathrm{L}}\colon L\to L_{j}$ denote the morphisms given by universal properties such that $p_{i,\mathrm{M}}v_{i,\mathrm{M}}=1_{M_{i}}$ and $q_{j,\mathrm{L}}u_{j,\mathrm{L}}=1_{L_{j}}$ for each $i$ and $j$.
\end{notation}
\begin{corollary}\label{zimmtech}Consider Notation \ref{techynotation}, let $a\colon G\to H$ be a morphism with $G,H\in\mathcal{S}$, and let $\varphi(v_{G})=(\exists u_{H}\colon v_{G}=u_{H}a)$. Then there is an isomorphism
\[
\begin{array}{c}
\kappa\langle\varphi\rangle\colon\prod_{i\in\mathbb{N}}\varphi(M_{i})\to\coprod_{j\in \mathtt{J}}\varphi(L_{j}),\,(f_{i}\mid i\in\mathbb{N})\mapsto(q_{j,\mathrm{L}}f\mid j\in \mathtt{J})
\end{array}
\]
where $f\colon G\to\prod_{i}M_{i}$ is given by the universal property, and whose inverse is 
\[
\begin{array}{c}
\kappa^{-1}\langle\varphi\rangle\colon \coprod_{j\in \mathtt{J}}\varphi(L_{j})\to\prod_{i\in\mathbb{N}}\varphi(M_{i}),\,(g_{j}\mid j\in \mathtt{J})\mapsto(\sum_{j\in \mathtt{J}}p_{i,\mathrm{M}}u_{j,\mathrm{L}}g_{j}\mid i\in \mathbb{N}).
\end{array}
\]
\end{corollary}
\begin{proof}By Lemma \ref{ppprops2} the restriction of $\gamma_{G,\mathrm{L}}$ and $\lambda_{G,\mathrm{M}}$ define isomorphisms 
\[
\begin{array}{cc}
\delta\colon\coprod_{j}\varphi(L_{j})\to \varphi(\coprod_{j}L_{j}), \mu\colon\varphi(\prod_{i}M_{i})\to\prod_{i}\varphi(M_{i}).
\end{array}
\]
Letting $\kappa\langle\varphi\rangle=\delta^{-1}\mu^{-1}$ and  $\kappa\langle\varphi\rangle^{-1}=\mu\delta$, the proof is straightforward.
\end{proof}
The proof of Lemma \ref{zimmtech2} follows the proof of the cited result of Huisgen-Zimmerman.
\begin{lemma}\label{zimmtech2}\emph{\cite[Lemma 4]{Hui1979}} Consider Notation \ref{techynotation} and let $\varphi_{1}(M)\supseteq \varphi_{2}(M) \supseteq \dots$ be a descending chain of pp-definable subgroups of $M$ of some sort $G\in\mathcal{S}$. Let 
\[
\begin{array}{cc}
\Psi(n)=\{\varphi_{n}(L_{j})\mid j\in \mathtt{J}\}, & \Pi(n)=\{\prod_{i<n}\varphi_{n}(M_{i}),\prod_{i\geq n}\varphi_{n}(M_{i})\},
\end{array}
\]
for any $n\in\mathbb{N}$, and for any $j$ consider the map $\rho_{n,j}=q_{j,\Psi(n)}\kappa\langle\varphi_{n}\rangle u_{\geq,\Pi(n)}$ given by
\[
\xymatrix{\prod_{i\geq n}\varphi_{n}(M_{i})\ar[rr]^{u_{\geq,\Pi(n)}} & & \prod_{i\in\mathbb{N}}\varphi_{n}(M_{i})\ar[r]^{\kappa\langle\varphi_{n}\rangle} & \coprod_{j\in J}\varphi_{n}(L_{j})\ar[rr]^{q_{j,\Psi(n)}} & & \varphi_{n}(L_{j}).
}
\]
For some $r\in\mathbb{N}$ and some $\mathtt{J}'\subseteq \mathtt{J}$ finite,  $\mathrm{im}(\rho_{r,j})\subseteq \varphi_{n}(L_{j})$ for all $n\geq r$ and $j\notin \mathtt{J}'$.
\end{lemma}
\begin{proof}
Considering Notation \ref{categnotation}, for each $n\in\mathbb{N}$, we have that $u_{\geq,\Pi(n)}$ is (the inclusion) given by sending $(x_{n+i}\mid i\in\mathbb{N})=(x_{n},x_{n+1},\dots)$ (where $x_{i}\in\varphi_{n}(M_{i})$) to the sequence $(0,\dots,0,x_{n},x_{n+1},\dots)$, the initial $n$ terms of which are $0$. Similarly, we have that $q_{j,\Psi(n)}$ is the restriction of $q_{j,\mathcal{T}(G,\mathrm{L})}$ for each $j$ and each $n$. 

We may assume $\mathtt{J}$ is non-empty since otherwise the statement is automatic. Assume for a contradiction that for any $r\in\mathbb{N}$ and any finite subset $\mathtt{J}'$ of $\mathtt{J}$, there exists $n\in\mathbb{N}$ with $n\geq r$ and there exists $j\in \mathtt{J}\setminus \mathtt{J}'$ such that $\mathrm{im}(\rho_{r,j})\not\subseteq \varphi_{n}(L_{j})$. Following \cite[Lemma 4]{Hui1979}, we firstly claim that there exists: a strictly increasing sequence of integers $r(0)<r(1)<r(2)<r(3)<\dots$; a sequence $j(0),j(1),j(2),\dots$ of pairwise distinct elements of $\mathtt{J}$ (so, where $j(n)\notin\{j(0),\dots,j(n-1)\}$ for all $n>0$); and a sequence of $\mathbb{N}$-tuples $\underline{m}_{0},\underline{m}_{1},\underline{m}_{2},\dots$ where
\[
\begin{array}{ccc}
\underline{m}_{d}\in\prod_{s\geq r(d)}\varphi_{r(d)}(M_{s}), & \rho_{r(d),j(d)}(\underline{m}_{d})\notin \varphi_{r(d+1)}(L_{j(d)}), & \rho_{r(d),j(d)}(\underline{m}_{t})=0,
\end{array}
\]
for all $d\in\mathbb{N}$ and all $t\in\mathbb{N}$ with $t<d$. We proceed inductively. Choose an element $j(-1)\in \mathtt{J}$. Let $\mathtt{J}'_{0}=\{j(-1)\}$ and $r(0)=1$. By our assumption that the conclusion is false, there exists an integer $r(1)>0$ and an element $j(0)\in \mathtt{J}\setminus\{j(-1)\}$ where $\mathrm{im}(\rho_{r(0),j(0)})\not\subseteq \varphi_{r(1)}(L_{j(0)})$. Hence there exists $\underline{n}_{0}\in\prod_{i\geq r(0)}\varphi_{r(0)}(M_{i})$ where $\rho_{r(0),j(0)}(\underline{n}_{0})\notin \varphi_{r(1)}(L_{j(0)})$, and so $r(1)>r(0)$.

We now iterate this process, yielding sequences $(r(d)\mid d\in\mathbb{N})$,  $(j(d)\mid d\in\mathbb{N})$ and  $(\underline{n}_{d}\mid d\in\mathbb{N})$ where $r(d)$ is a positive integer, $j(d)\in \mathtt{J}$ and $\underline{n}_{d}\in\prod_{i\geq r(d)}\varphi_{r(d)}(M_{i})$ and such that $r(d+1)>r(d)$, $j(d+1)\in J\setminus\{j(-1),\dots,j(d)\}$, and $\rho_{r(d),j(d)}(\underline{n}_{d})\notin \varphi_{r(d+1)}(L_{j(d)})$. Now fix $t\in\mathbb{N}$. Since we are considering coproducts of abelian groups, note that we have $\rho_{r(t),j}(\underline{n}_{t})=0$ for all but finitely many $j\in \mathtt{J}$. For each $j$ define the map $l_{j}^{t}\colon G\to L_{j}$ in $\varphi_{r(t)}(L_{j})$ by $l_{j}^{t}=\rho_{r(t),j}(\underline{n}_{t})$ if $j=j(d)$ for some $d>t$, and $l_{j}^{t}=0$ otherwise. Now let 
\[
\begin{array}{c}
\tilde{\underline{m}}_{d}=u_{\geq,\Pi(r(d))}(\underline{n}_{d})-\kappa^{-1}\langle\varphi_{r(d)}\rangle(l^{d}_{j}\mid j\in \mathtt{J})\in\prod_{i\in\mathbb{N}}\varphi_{r(d)}(M_{i}).
\end{array}
\]
By construction we have $p_{i,\mathrm{M}}(\tilde{\underline{m}}_{d})=0$ for all $i<r(d)$, and if  $t<d$ then
\[
\begin{array}{c}
q_{j(d),\mathcal{T}(G,\mathrm{L})}(\kappa\langle\varphi_{r(d)}\rangle(\tilde{\underline{m}}_{t}))=q_{j(d),\mathcal{T}(G,\mathrm{L})}(\kappa\langle\varphi_{r(d)}\rangle(u_{\geq,\Pi(r(d))}(\underline{n}_{t})))-l^{t}_{j(d)}=0.
\end{array}
\]
Now, writing $\tilde{\underline{m}}_{d}=(m_{d,0},m_{d,1},\dots)$ where $m_{d,i}\in\varphi_{r(d)}(M_{i})$ for all $i\in\mathbb{N}$, the above gives $m_{d,i}=0$ for all $i<r(d)$, so we may define 
\[
\begin{array}{c}
\underline{m}_{d}=(m_{d,r(d)+i}\mid i\in \mathbb{N})=(m_{d,r(d)},m_{d,r(d)+1},m_{d,r(d)+2}\dots)\in\prod_{s\geq r_{d}}\varphi_{r(d)}(M_{s}).
\end{array}
\]
So we have $\tilde{\underline{m}}_{d}=u_{\geq,\Pi(r(d))}(\underline{m}_{d})$, and  therefore
\[
\begin{array}{c}
\rho_{r(d),j(d)}(\underline{m}_{d})=q_{j(d),\mathcal{T}(G,\mathrm{L})}(\kappa\langle\varphi_{r(d)}\rangle(\tilde{\underline{m}}_{d}))=
\rho_{r(d),j(d)}(\underline{n}_{d})-l^{d}_{j(d)}\notin \varphi_{r(d+1)}(L_{j(d)}),
\end{array}
\]
since by definition $l^{d}_{j(d)}=0$. Our calculations above likewise show that if $t<d$ then we have $\rho_{r(d),j(d)}(\underline{m}_{t})=0$. This verifies our initial claim. Now let $d$ vary. Since $r(d)<r(d+1)$ for all $d$, for each $i\in\mathbb{N}$ observe that there are finitely many $d$ with $r(d)\leq i$. Hence the sum
\[
\begin{array}{c}
\tilde{\underline{m}}=\sum_{i\in\mathbb{N}}\tilde{\underline{m}}_{i}=(\sum_{d\in\mathbb{N},r(d)\leq i}m_{d,i}\mid i\in\mathbb{N})\in\prod_{i\in\mathbb{N}}\mathcal{T}(G,M_{i})
\end{array}
\]
is well-defined. Now, for each $l\in\mathbb{N}$, by combining everything so far with Corollary \ref{zimmtech}, we have 
\[\begin{array}{c}
q_{j(l),\mathcal{T}(G,\mathrm{L})}(\gamma^{-1}_{G,\mathrm{L}}(\lambda^{-1}_{G,\mathrm{M}}(\tilde{\underline{m}})))=\sum_{i\in\mathbb{N}}q_{j(l),\mathcal{T}(G,\mathrm{L})}(\gamma^{-1}_{G,\mathrm{L}}(\lambda^{-1}_{G,\mathrm{M}}(\tilde{\underline{m}}_{i})))\\=\sum_{i\in\mathbb{N}}q_{j(l),\mathcal{T}(G,\mathrm{L})}(\kappa\langle\varphi_{r(i)}\rangle(\tilde{\underline{m}}_{i}))=\sum_{i\geq l}q_{j(l),\mathcal{T}(G,\mathrm{L})}(\kappa\langle\varphi_{r(i)}\rangle(\tilde{\underline{m}}_{i}))\\=\sum_{i\geq l}\rho_{r(i),j(l)}(\underline{m}_{i})=\rho_{r(l),j(l)}(\underline{m}_{l})+\sum_{i>l}\rho_{r(i),j(l)}(\underline{m}_{i}).
\end{array}
\]
Now recall that $\rho_{r(l),j(l)}(\underline{m}_{l})\notin \varphi_{r(l+1)}(L_{j(l)})$. Since $\varphi_{1}(M)\supseteq\varphi_{2}(M)\supseteq\dots$ is descending and $r(i+1)>r(i)$ for all $i$, we have $\rho_{r(i),j(l)}(\underline{m}_{i})\in \varphi_{r(i)}(L_{j(l)})\subseteq \varphi_{r(l+1)}(L_{j(l)})$ whenever $i>l$. Together with the above, this shows \[q_{j(l),\mathcal{T}(G,\mathrm{L})}(\gamma^{-1}_{G,\mathrm{L}}(\lambda^{-1}_{G,\mathrm{M}}(\tilde{\underline{m}})))\neq 0\] for all $l$. Since $\gamma^{-1}_{G,\mathrm{L}}(\lambda^{-1}_{G,\mathrm{M}}(\tilde{\underline{m}}))$ lies in the coproduct $\coprod_{j}\mathcal{T}(G,L_{j})$, and therefore must have had finite support over $j\in \mathtt{J}$, we have a contradiction, since the set $\{j(l)\mid l\in\mathbb{N}\}$ is in bijection with $\mathbb{N}$.
\end{proof} 
\section{Annihilator subobjects and pp-definable subgroups.}\label{4}
Recall that, in a category with all small coproducts, a set $\{\mathscr{G}_{\alpha}\mid\alpha\in \mathtt{X}\}$ of objects is called a set of \textit{generators} provided, for each object $\mathscr{Q}$, there is an epimorphism $\coprod_{\alpha}\mathscr{G}_{\alpha}\to\mathscr{Q}$. In case $ \mathtt{X}$ is a singleton we say the category \textit{has a generator}. Recall an additive category $\mathcal{A}$ is \textit{Grothendieck} provided: $\mathcal{A}$ is abelian; $\mathcal{A}$ has all small coproducts; $\mathcal{A}$ has a generator; and the direct limit of any short exact sequence in $\mathcal{A}$ is again exact. 
\begin{remark}\label{locallycoherentremark}
Let $\mathcal{A}$ be a Grothendieck category. Recall that an object $\mathscr{Q}$ of $\mathcal{A}$ is \textit{finitely presented} provided the functor $\mathcal{A}(\mathscr{Q},-)\colon\mathcal{A}\to\mathbf{Ab}$ commutes with direct limits. Recall that an object $\mathscr{S}$ of $\mathcal{A}$ is \textit{finitely generated} provided there is an exact sequence $\mathscr{R}\to\mathscr{Q}\to 0$ in $\mathcal{A}$. The categories considered both in work of Garcia and Dung \cite{GarDun1994} and in work of Harada \cite{Har1973} were Grothendieck categories with a set of finitely generated generators.

Following Krause \cite{Kra1997}, a category $\mathcal{A}$ is said to be \textit{locally coherent} provided: $\mathcal{A}$ is a Grothendieck category; $\mathcal{A}$ has a set $\{\mathscr{G}_{\alpha}\mid\alpha\in \mathtt{X}\}$ of generators such that each $\mathscr{G}_{\alpha}$ is finitely presented; and the full subcategory of $\mathcal{A}$ consisting of finitely presented objects is abelian. As noted at the top of \cite[p.3]{GarPre2005}, $\mathbf{Mod}\text{-}\mathcal{T}^{c}$ is locally coherent.
\end{remark}
Thus, in Definition \ref{defannsub} and Lemmas \ref{annsubmodT} and \ref{annsubmodTdcc}, we specify various definitions and results from  \cite{GarDun1994} and \cite{Har1973} to $\mathbf{Mod}\text{-}\mathcal{T}^{c}$ (which can be done by Remark \ref{locallycoherentremark}). We now recall a notion introduced by Harada.
\begin{definition}\label{defannsub}\cite[\S 1]{Har1973} Let $\mathcal{A}$ be a Grothendieck category with a set $\{\mathscr{G}_{\alpha}\mid\alpha\in \mathtt{X}\}$ of finitely generated generators. Let $\mathscr{Q}$ and $\mathscr{R}$ be objects in $\mathcal{A}$. A subobject $\mathscr{P}$ of $\mathscr{Q}$ is said to be an $\mathscr{R}$-\textit{annihilator subobject} of $\mathscr{Q}$ provided $\mathscr{P}=\bigcap_{\mathfrak{f}\in \mathtt{K}} \mathrm{ker}(\mathfrak{f})$ where the intersection is taken over morphisms $\mathfrak{f}\colon\mathscr{Q}\to\mathscr{R}$ running through some $\mathtt{K}\subseteq\mathcal{A}(\mathscr{Q},\mathscr{R})$.
\end{definition}
Lemma \ref{annsubmodT} focuses on a particular context of Definition \ref{defannsub}. That is, we specifiy to the locally coherent category $\mathbf{Mod}\text{-}\mathcal{T}^{c}$ and consider the image of objects under $\mathbf{Y}$. Note Lemma \ref{annsubmodT} was written only to simplify the proof of Lemma \ref{annsubmodTdcc}.
\begin{lemma}\label{annsubmodT}
Let $\mathfrak{q}\colon\mathbf{Y}(X)\to \mathscr{Q}$ be an epimorphism in $\mathbf{Mod}\text{-}\mathcal{T}^{c}$ where the object $X$ of $\mathcal{T}$ is compact. If $M$ and $Z$ are objects of $\mathcal{T}$ and $\mathcal{T}^{c}$ respectively, then for any $\mathbf{Y}(M)$-annihilator subobject $\mathscr{P}=\bigcap_{\mathfrak{f}\in\mathtt{K}}\mathrm{ker}(\mathfrak{f})$ of $\mathscr{Q}$ \emph{(}where $\mathtt{K}\subseteq\mathcal{A}(\mathscr{Q},\mathscr{R})$\emph{)} we have
\[\mathscr{P}(Z)=\{\mathfrak{q}_{Z}(g)\mid g\in\mathcal{T}(Z,X)\text{ and }\mathfrak{f}_{X}(\mathfrak{q}_{X}(1_{X}))g=0\text{ for all }\mathfrak{f}\in\mathtt{K}\}.
\]\end{lemma}
\begin{proof}It suffices to assume $\mathtt{K}\neq\emptyset$. Let $p\in\mathscr{Q}(Z)$. Since $\mathfrak{q}_{Z}$ is onto, $p=\mathfrak{q}_{Z}(g)$ for some $g\in\mathcal{T}(Z,X)$.
Since $\mathfrak{q}$ and each $\mathfrak{f}$ are morphisms in $\mathbf{Mod}\text{-}\mathcal{T}^{c}$ the diagram of abelian groups given by
\[
\xymatrix@R=1.5pc{\mathcal{T}(X,X)\ar[d]_{\mathcal{T}(g,X)}\ar[r]_{\mathfrak{q}_{X}} & \mathscr{Q}(X)\ar[d]^{\mathscr{Q}(g)}\ar[r]_{\mathfrak{f}_{X}} & \mathcal{T}(X,M)\ar[d]^{\mathcal{T}(g,M)}\\
\mathcal{T}(Z,X)\ar[r]^{\mathfrak{q}_{Z}} & \mathscr{Q}(Z)\ar[r]^{\mathfrak{f}_{Z}} & \mathcal{T}(Z,M)}
\]
commutes. By the commutativity of the diagram we have $\mathfrak{q}_{Z}(g)=\mathscr{Q}(g)(\mathfrak{q}_{X}(1_{X}))$ and (hence) $\mathfrak{f}_{Z}(p)=\mathfrak{f}_{X}(\mathfrak{q}_{X}(1_{X}))g$. Now suppose $p\in\bigcap_{\mathfrak{f}\in\mathtt{K}}\mathrm{ker}(\mathfrak{f}_{Z})$. By the above this means $\mathfrak{f}_{X}(\mathfrak{q}_{X}(1_{X}))g=0$ for all $\mathfrak{f}$. Hence $\mathscr{P}(Z)$ lies in the ride hand side of the required equality. The reverse inclusion is straightforward.
\end{proof}
Lemma \ref{annsubmodTdcc} is based on a proof of a given by Huisgen-Zimmerman \cite[Corollary 7]{Hui2000} of a well-known characterisation of  $\Sigma$-\textit{injective} modules due to Faith \cite[Proposition 3]{Fai1966}. We use Lemma \ref{annsubmodTdcc} to simplify the proof of Lemma \ref{sigmathendcc}, a key result employed in the sequel. 
\begin{lemma}\label{annsubmodTdcc}
Let $M$ and $X$ be objects in $\mathcal{T}$ and $\mathcal{T}^{c}$ respectively, and let $\mathcal{T}(-,X)\to \mathscr{Q}\to0$ be a sequence in $\mathbf{Mod}\text{-}\mathcal{T}^{c}$ which is exact. Any strictly ascending chain of \emph{(}$\mathbf{Y}(M)$-annihilator subobjects of $\mathscr{Q}$\emph{)} gives a strictly descending chain of \emph{(}pp-definable subgroups of $M$ of sort $X$\emph{)}.
\end{lemma}
\begin{proof}Suppose $\mathscr{P}_{1}\subsetneq\mathscr{P}_{2}\subsetneq\dots$ is a strictly ascending chain of $\mathbf{Y}(M)$-annihilator subobjects of $\mathscr{Q}$, say where, for each integer $n>0$, we have $\mathscr{P}_{n}=\bigcap_{\mathfrak{f}\in\mathtt{K}[n]}\mathrm{ker}(\mathfrak{f})$ for some subset $\mathtt{K}[n]$ of morphisms in $\mathbf{Mod}\text{-}\mathcal{T}^{c}$ of the form $\mathfrak{f}\colon\mathscr{Q}\to\mathbf{Y}(M)$. 

We assume $\mathtt{K}[1]\supsetneq\mathtt{K}[2]\supsetneq\dots$ without loss of generality. For each $n$ there is an object $Z_{n}$ of $\mathcal{T}^{c}$ for which $\mathscr{P}_{n}(Z_{n})\subsetneq\mathscr{P}_{n+1}(Z_{n})$, and we choose $h_{n}\in\mathscr{P}_{n+1}(Z_{n})\setminus\mathscr{P}_{n}(Z_{n})$. Let $\mathfrak{q}\colon\mathcal{T}(X,-)\to \mathscr{Q}$ be the epimorphism in $\mathbf{Mod}\text{-}\mathcal{T}^{c}$ giving the exact sequence $\mathcal{T}(-,X)\to \mathscr{Q}\to0$. Since $\mathfrak{q}_{Z_{n}}$ is onto and $h_{n}\in\mathscr{Q}(Z_{n})$ we have $h_{n}=\mathfrak{q}_{Z_{n}}(g_{n})$ for some morphism $g_{n}\colon Z_{n}\to X$. By Lemma \ref{annsubmodT}, since $h_{n}\in\mathscr{P}_{n+1}(Z_{n})$ we have that $\mathfrak{f}_{X}(\mathfrak{q}_{X}(1_{X}))g_{n}=0$ for all $\mathfrak{f}\in\mathtt{K}[n+1]$. Similarly, since $h_{n}\notin\mathscr{P}_{n}(Z_{n})$ there exists $\mathfrak{s}(n)\in\mathtt{K}[n]$ such that $\mathfrak{s}(n)_{X}(\mathfrak{q}_{X}(1_{X}))g_{n}\neq0$. 

We now follow the proof of \cite[Proposition 3.1]{GarPre2005}. Since $\mathcal{T}^{c}$ is triangulated any morphism $g:Z\to X$ in $\mathcal{T}^{c}$ yields a triangle in $\mathcal{T}^{c}$ and an exact sequence in $\mathbf{Ab}$
\[
\begin{array}{cc}
\xymatrix{
Z\ar[r]^{g} & X\ar[r]^{b} & Y\ar[r] & \Sigma Z}, & \xymatrix{
\mathcal{T}(Y,W)\ar[r]^{-b} & \mathcal{T}(X,W)\ar[r]^{-g} & \mathcal{T}(Z,W),}
\end{array}
\]
given by applying the contravaraint functor $\mathcal{T}(-,W)\colon\mathcal{T}\to\mathbf{Ab}$. In other words, $b$  is a \textit{pseudocokernel} of $g$, and so any morphism $t\colon X\to W$ in $\mathcal{T}^{c}$ with $tg=0$ must satisfy  $t=sb$ for some $s\colon Y\to Z$. 

Let $t_{n}=\mathfrak{s}(n)_{X}(\mathfrak{q}_{X}(1_{X}))$ for each $n$. Combining what we have so far, for each $n$ we have $\mathfrak{s}(n+1)\in\mathtt{K}[n+1]$, so $t_{n+1}g_{n}=0$, and so $t_{n+1}=s_{n}b_{n}$ for some morphism $s_{n}\colon Y_{n}\to M$, and so $t_{n+1}\in M b_{n}$. On the other hand, if $t_{n+1}\in M b_{n+1}$ then $t_{n+1}g_{n+1}=0$ which contradicts that $\mathfrak{s}(n+1)_{X}(\mathfrak{q}_{X}(1_{X}))g_{n+1}\neq0$, and so $t_{n+1}\notin M b_{n+1}$. This gives a strict descending chain
\[
\begin{array}{c}
Mb_{1}\supsetneq Mb_{1} \cap Mb_{2} \supsetneq Mb_{1} \cap Mb_{2}\cap Mb_{3} \supsetneq \dots \supsetneq \bigcap_{i=1}^{d}Mb_{i}\supsetneq \dots 
\end{array}
\]
A direct application of \cite[Proposition 3.1]{GarPre2005} shows that each finite intersection $\bigcap_{i=1}^{d}Mb_{i}$ has the form $Ma_{d}$ for some morphism in $\mathcal{T}^{c}$ of the form $a_{d}\colon X\to W_{d}$. So the chain above is, as required, a strictly descending chain of pp-definable subgroups of $M$ of sort $X$.
\end{proof}
\begin{definition}\cite[\S 1]{Har1973} Let $\mathcal{A}$ be a Grothendieck category with a set of finitely generated generators. Fix an object $\mathscr{M}$ of $\mathcal{A}$. We say that $\mathscr{M}$ is $\Sigma$-\textit{injective} if, for any set $\mathtt{I}$, the coproduct $\mathscr{M}^{(\mathtt{I})}=\coprod_{i\in \mathtt{I}}\mathscr{M}$ is injective. We say that $\mathscr{M}$ is \textit{fp}-\textit{injective} if, whenever $0\to\mathscr{P}\to\mathscr{R}\to\mathscr{Q}\to 0$ is an exact sequence in $\mathcal{A}$ where $\mathscr{Q}$ is finitely presented, any morphism $\mathscr{P}\to\mathscr{M}$ extends to a morphism $\mathscr{R}\to\mathscr{M}$; see \cite[\S 1]{GarDun1994}.
\end{definition}
For the proof of Corollary \ref{sigmathendccconv} we recall two results: Proposition \ref{fpinjgd}, due to Garcia and Dung, characterises $\Sigma$-injectivity in the fp-injective setting; and Lemma \ref{fpinjkra}, due to Krause, shows that it is sufficient to consider the fp-injective setting.
\begin{proposition}\emph{\cite[Proposition 1.3]{GarDun1994}}\label{fpinjgd} Let $\mathscr{M}$ be an fp-injecitve in a Grothendieck category $\mathcal{A}$ which has a set $\mathtt{G}$ of finitely presented generators $\mathscr{G}$. Then $\mathscr{M}$ is $\Sigma$-injective if and only if, for each $\mathscr{G}\in\mathtt{G}$, every ascending chain of $\mathscr{M}$-annihilator subobjects of $\mathscr{G}$ must stabilise.
\end{proposition}
\begin{lemma}\emph{\cite[Lemma 1.6]{Kra2000}}\label{fpinjkra} For any $M$ in $\mathcal{T}$ the image $\mathbf{Y}(M)$ is fp-injective.
\end{lemma}
\begin{corollary}\label{sigmathendccconv}Let $M$ be an object in $\mathcal{T}$ such that for any compact object $X$ of $\mathcal{T}$ we have that each descending chain $Ma_{1}\supseteq Ma_{2}\supseteq \dots$ of pp-definable subgroups of $M$ of sort $X$ must stabilise. Then the image $\mathbf{Y}(M)$ of $M$ in $\mathbf{Mod}\text{-}\mathcal{T}^{c}$ is $\Sigma$-injective.
\end{corollary}
\begin{proof}We prove the contrapositive, so we assume $\mathbf{Y}(M)$ is not $\Sigma$-injective. Recall, from Remark \ref{locallycoherentremark}, that $\mathbf{Mod}\text{-}\mathcal{T}^{c}$ is locally coherent, and so it is a Grothendieck category with a set $\mathtt{G}$ of finitely presented generators. 

Note $\mathscr{M}=\mathbf{Y}(M)$ is fp-injective by Lemma \ref{fpinjkra}, and combining our initial assumption with Proposition \ref{fpinjgd} shows that, for some $\mathscr{G}\in\mathtt{G}$, there exists a strictly ascending chain of $\mathscr{M}$-annihilator subobjects of $\mathscr{G}$. Since $\mathscr{G}$ is finitely presented, there is an exact sequence of the form $\mathcal{T}(-,Y)\to\mathcal{T}(-,X)\to\mathscr{G}\to 0$ in $\mathbf{Mod}\text{-}\mathcal{T}^{c}$ where $X$ and $Y$ lie in $\mathcal{T}^{c}$. 

By Lemma \ref{annsubmodTdcc} the aforementioned ascending chain strict ascending chain gives rise to a strictly descending chain of pp-definable subgroups of $M$ of sort $X$. 
\end{proof}
\section{$\Sigma$-pure-injective objects and canonical morphisms.}\label{4-1}
\begin{definition}
Recall Notation \ref{categnotation}. Let $\mathtt{I}$ be a set and let $M$ be an object of $\mathcal{T}$. By the universal properties of the product and coproduct of the collection $\mathrm{M}=\{M\mid i\in \mathtt{I}\}$, there exists a unique \textit{summation morphism} $\sigma_{\mathtt{I},\mathrm{M}}:\coprod_{i} M \to M$ and a unique \textit{canonical morphism} $\iota_{\mathtt{I},\mathrm{M}}:\coprod_{i} M\to \prod_{i} M$ for which $\sigma_{\mathtt{I},\mathrm{M}}u_{i,\mathrm{M}}=1_{M}$ and $\iota_{\mathtt{I},\mathrm{M}}u_{i,\mathrm{M}}=v_{i,\mathrm{M}}$ for each $i$.
\end{definition}
\begin{proposition}\label{sigmapimapsdef} Let $M$ be an object of $\mathcal{T}$ and let $\mathtt{I}$ be a set. Then the canonical morphism  $\iota_{\mathtt{I},\mathrm{M}}$ is a pure monomorphism.
\end{proposition}
\begin{proof}Let $X$ be a object in $\mathcal{T}$ which is compact. In general: the morphism $\lambda_{X,\mathrm{M}}$ is an isomorphism; the canonical morphism $\iota_{\mathtt{I},\mathcal{T}(X,\mathrm{M})}$ is injective; and $\iota_{\mathtt{I},\mathcal{T}(X,\mathrm{M})}$ is the composition $\lambda_{X,\mathrm{M}}\mathcal{T}(X,\iota_{\mathtt{I},\mathrm{M}})\gamma_{X,\mathrm{M}}$. 

Since $X$ is compact the morphism $\gamma _{X,\mathrm{M}}$ is an isomorphism. This shows $\mathcal{T}(X,\iota_{\mathtt{I},\mathrm{M}})$ is injective if $X$ is compact, and so $\iota_{\mathtt{I},\mathrm{M}}$ is a pure monomorphism.
\end{proof}

\begin{definition}\cite[Definition 1.1]{Kra2000} An object $M$ of $\mathcal{T}$ is called \textit{pure}-\textit{injective} if each pure monomorphism $M\to N$ is a section, and $M$ is called $\Sigma$-\textit{pure}-\textit{injective} if, for any set $\mathtt{I}$, the coproduct $\coprod_{i\in \mathtt{I}}M=M^{(\mathtt{I})}$ is pure-injective.
\end{definition}
At this point it is worth recalling some characterisations of purity due to Krause. Theorem \ref{krachar} is analogous to \cite[Theorem 7.1 (ii,v,vi)]{JenLen1989}.
\begin{theorem}\label{krachar2}\label{krachar}
\emph{\cite[Theorem 1.8, (1,3,5)]{Kra2002}} For an object $M$ of $\mathcal{T}$ the following statements are equivalent.
\begin{enumerate}
\item The object $M$ of $\mathcal{T}$ is pure-injective.
\item The object $\mathbf{Y}(M)$ of $\mathbf{Mod}\text{-}\mathcal{T}^{c}$ is injective. 
\item For any set $\mathtt{I}$ the morphism $\sigma_{\mathtt{I},\mathrm{M}}$ factors through the morphism $\iota_{\mathtt{I},\mathrm{M}}$.
\end{enumerate}
\end{theorem}
Proposition \ref{sigmapimapsdef2} is analogous to parts (i) and (ii) in \cite[Theorem 8.1]{JenLen1989}.
\begin{proposition}\label{sigmapimapsdef2} An object $M$ of $\mathcal{T}$ is $\Sigma$-pure-injective if and only if, for each set $\mathtt{I}$, the canonical morphism $\iota_{\mathtt{I},\mathrm{M}}$ is a section. 
\end{proposition}
\begin{proof}Assume that $M$ is $\Sigma$-pure-injective and that $\mathtt{I}$ is a set. By assumption the domain of  $\iota_{\mathtt{I},\mathrm{M}}$ is pure-injective. Since $\iota_{\mathtt{I},\mathrm{M}}$ is a pure monomorphism, this means it is a section. Supposing conversley that $\iota_{\mathtt{I},\mathrm{M}}$ is a section for each set $\mathtt{I}$, it remains to show that $M$ is $\Sigma$-pure-injective. Choose a set $\mathtt{T}$ and let $N=\coprod_{t\in \mathtt{T}}M$. It suffices to prove $N$ is pure-injective. Let $\mathtt{S}$ be any set and consider the collection $\mathrm{N}=\{N\mid s\in \mathtt{S}\}$. By Theorem \ref{krachar} it suffices to find a map $\theta_{\mathtt{S},\mathrm{N}}\colon\prod_{s}N\to N$ such that $\sigma_{\mathtt{S},\mathrm{N}}=\theta_{\mathtt{S},\mathrm{N}}\iota_{\mathtt{S},\mathrm{N}}$. Let $\mathrm{M}=\{M\mid t\in \mathtt{T}\}$. 

For each $(s,t)\in \mathtt{S}\times \mathtt{T}$ the morphisms $u_{s,\mathrm{N}}u_{t,\mathrm{M}}$ satisfy the universal property of the coproduct $\coprod_{s,t}M$, and so we assume $u_{s,t,\mathrm{M}}=u_{s,\mathrm{N}}u_{t,\mathrm{M}}$ without loss of generality. Consider the morphisms $\varphi_{s,t,\mathrm{M}}=q_{t,\mathrm{M}}p_{s,\mathrm{N}}$ for each $(s,t)\in \mathtt{S}\times \mathtt{T}$. Since $u_{s,t,\mathrm{M}}=u_{s,\mathrm{N}}u_{t,\mathrm{M}}$ we have $q_{s,t,\mathrm{M}}=q_{t,\mathrm{M}}q_{s,\mathrm{N}}$ by uniqueness. Consequently $\varphi_{s,t,\mathrm{M}}v_{s,\mathrm{N}}q_{s,\mathrm{N}}u_{s,t,\mathrm{M}}$ is the identity on $M$. By the universal property of the product, there is a morphism $ \Xi\colon\prod_{s}N\to \prod_{s,t}M$ such that $p_{s,t,\mathrm{M}} \Xi=\varphi_{s,t,\mathrm{M}}$ for each $(s,t)$. 
It suffices to let $\theta_{\mathtt{S},\mathrm{N}}=\sigma_{\mathtt{S},\mathrm{N}}\pi_{\mathtt{S}\times \mathtt{T},\mathrm{M}}\Xi$. By the uniqueness of the involved morphisms, it is straightforward to see that $\sigma_{\mathtt{S},\mathrm{N}}=\theta_{\mathtt{S},\mathrm{N}}\iota_{\mathtt{S},\mathrm{N}}$.

\end{proof}
Lemma \ref{sigmathendcc} is analogous to \cite[Theorem 8.1(ii,iii)]{JenLen1989}.
\begin{lemma}\label{sigmathendcc}If $G\in\mathcal{S}$ and $M$ is a $\Sigma$-pure-injective object in $\mathcal{T}$ then every descending chain of pp-definable subgroups of $M$ of sort $G$ stabilises. 
\end{lemma}
\begin{proof}For a contradiction we assume the existence of a strictly descending chain $Ma_{0}\supsetneq Ma_{1}\supsetneq Ma_{2}\supsetneq \dots$ of (pp-definable subgroups of $M$ of sort $G$) for some $G\in\mathcal{S}$. Hence there is a collection of compact objects $H_{n}\in\mathcal{S}$ such that $a_{n}\in\mathcal{T}(G,H_{n})$ for each $n\in\mathbb{N}$. By our assumption we may choose elements $b_{n}\in\mathcal{T}(H_{n},M)$ such that $b_{n}a_{n}\notin Ma_{n+1}$. By Proposition \ref{sigmapimapsdef2} the canonical morphism $\iota_{\mathbb{N},\mathrm{M}}:\coprod_{\mathbb{N}}M\to\prod_{\mathbb{N}}M$ is a section, and so there is some morphism $\pi_{\mathbb{N},\mathrm{M}}\colon\prod_{\mathbb{N}}M\to\coprod_{\mathbb{N}}M$ such that $\pi_{\mathbb{N},\mathrm{M}}\iota_{\mathbb{N},\mathrm{M}}$ is the identity on $\coprod_{\mathbb{N}}M$. 
After applying the functor $\mathcal{T}(G,-)$ (from $\mathcal{T}$ to the category of abelian groups) this means $\mathcal{T}(G,\pi_{\mathbb{N},\mathrm{M}})\mathcal{T}(G,\iota_{\mathbb{N},\mathrm{M}})$ is the identity on $\mathcal{T}(G,\coprod_{\mathbb{N}}M)$. Let $\underline{ba}=(b_{n}a_{n}\mid n\in \mathbb{N})\in\prod_{n}\mathcal{T}(G,M)$.  Fix $n\in\mathbb{N}$ and let $\varphi_{n}(v_{G})$ be the formula $(\exists u_{H_{n}} \colon v_{G}=u_{H_{n}}a_{n})$. Let $\mathrm{M}=\{M\mid n\in\mathbb{N}\}$. Recall $\lambda_{G,\mathrm{M}}$ is always an isomorphism, and since $G$ is compact, $\gamma_{G,\mathrm{M}}$  is an isomorphism. Define $\mu$ by
\[\begin{array}{c}
\mu=(\gamma_{G,\mathrm{M}})^{-1}\mathcal{T}(G,\pi_{\mathbb{N},\mathrm{M}})(\lambda_{G,\mathrm{M}})^{-1}\colon\prod_{n\in\mathbb{N}}\mathcal{T}(G,M)\to\coprod_{n\in\mathbb{N}}\mathcal{T}(G,M).
\end{array}
\]
Let $\mu(\underline{ba})=(c_{n}\mid n\in\mathbb{N})$. The contradiction we will find is that $c_{l}\neq 0$ for all $l\in\mathbb{N}$, which contradicts that $\mu$ has codomain $\coprod_{n\in\mathbb{N}}\mathcal{T}(G,M)$. Fix $l\in\mathbb{N}$. Now let
\[
\underline{ba}_{\leq l}=(b_{0}a_{0},\dots,b_{l}a_{l},0,0,\dots)\text{ and }\underline{ba}_{>l}=(0,\dots,0,b_{l+1}a_{l+1},b_{l+2}a_{l+2},\dots ),
\] where the first $l$ entries of $\underline{ba}_{>l}$ are $0$. 

Note that $\underline{ba}_{\leq l}\in \coprod_{n}\mathcal{T}(G,M)$. Furthermore, since the chain $Ma_{0}\supseteq Ma_{1}\supseteq \cdots$ is descending, we have $b_{n}a_{n}\in\varphi_{l}(M)$ for all $n>l$ and so $\underline{ba}_{>l}\in\prod_{n}\varphi_{l+1}(M)$. By Lemma \ref{ppprops}(ii), the restrictions of $(\lambda_{G,\mathrm{M}})^{-1}$ and $(\gamma_{G,\mathrm{M}})^{-1}$ respectively define isomorphisms $\prod_{n}\varphi_{l+1}(M)\to \varphi_{l+1}(\prod_{n}M)$ and $\varphi_{l+1}(\coprod_{n}M)\to \coprod_{n}\varphi_{l+1}(M)$. Similarly $\mathcal{T}(G,\pi_{\mathbb{N},\mathrm{M}})$ restricts to define a morphism $\varphi_{l+1}(\prod_{n}M)\to\varphi_{l+1}(\coprod_{n}M)$. Altogether we have that $ \mu$ restricts to a morphism $\prod_{n}\varphi_{l+1}(M)\to\coprod_{n}\varphi_{l+1}(M)$. Let $\mu(\underline{ba}_{>l})=(d_{n}\mid n\in\mathbb{N})$, and so $d_{n}\in\varphi_{l+1}(M)$ for all $n$. 

Recall it suffices to show $c_{l}\neq0$ where $\mu(\underline{ba})=(c_{n}\mid n\in\mathbb{N})$. From the above,
\[
\begin{array}{c}(c_{0},\dots,c_{l},c_{l+1},\dots)=\mu (\underline{ba})=\mu(\underline{ba}_{\leq l} + \underline{ba}_{>l})=\underline{ba}_{\leq l} + \mu(\underline{ba}_{>l})\\=(b_{0}a_{0}+d_{0},\dots,b_{l}a_{l}+d_{l},b_{l+1}a_{l+1}+d_{l+1},\dots),
\end{array}
\]
and so $c_{l}\neq0$ as otherwise $\varphi_{l+1}(M)\ni -d_{l}=b_{l}a_{l}\notin \varphi_{l+1}(M)$. 
\end{proof}

We now recall a result of Krause which is used heavily in the sequel.
\begin{corollary}\emph{\cite[Corollary 1.10]{Kra2000}}\label{kraequivpi} The functor $\mathbf{Y}$ gives an equivalence between the full subcategory of $\mathcal{T}$ consisting of pure-injective objects and the full subcategory of  $\mathbf{Mod}\text{-}\mathcal{T}^{c}$ consitsting of injective objects. 
\end{corollary}
Note that, since $\mathbf{Y}$ is additive, by Corollary \ref{kraequivpi} a pure-injective object $M$ is indecomposable if and only if $\mathbf{Y}(M)$ is an indecomposable (injective) object.
\begin{remark}\label{yonedacoprod}
 Fix a collection $\mathrm{M}=\{M_{i}\mid i\in \mathtt{I}\}$ of objects in $\mathcal{T}$.  Since small coproducts exist in $\mathcal{T}$ (by Assumption \ref{ass22}) and $\mathbf{Ab}$, the morphisms $\gamma_{X,\mathrm{M}}$ combine to define a natural transformation $\coprod _{i}\mathcal{T}(-, M_{i})\to \mathcal{T}(-,\coprod _{i}M_{i})$. By Definition \ref{compactob} this transformation is in fact an isomorphism $\coprod _{i}\mathbf{Y}(M_{i})\simeq\mathbf{Y}(\coprod _{i}M_{i})$ in $\mathbf{Mod}\text{-}\mathcal{T}^{c}$ which shows that $\mathbf{Y}$ preserves small coproducts.
\end{remark}
\begin{corollary}\label{sigmaequiv}The functor $\mathbf{Y}$ gives an equivalence between the full subcategory of $\mathcal{T}$ consisting of $\Sigma$-pure-injective objects and the full subcategory of  $\mathbf{Mod}\text{-}\mathcal{T}^{c}$ consitsting of $\Sigma$-injective objects. 
\end{corollary}
\begin{proof}By definition, an object $M$ of $\mathcal{T}$ is $\Sigma$-pure-injective if and only if, for every set $\mathtt{I}$, the coproduct $M^{(\mathtt{I})}$ is pure-injective. By Corollary \ref{kraequivpi} and Remark \ref{yonedacoprod}, given any such $\mathtt{I}$, $M^{(\mathtt{I})}$ is pure-injective if and only if $\mathbf{Y}(M^{(\mathtt{I})})\simeq(\mathbf{Y}(M))^{(\mathtt{I})}$ is injective. 

This shows that  $\mathbf{Y}$ induces a functor from the full subcategory of $\Sigma$-pure-injectives in $\mathcal{T}$ and the full subcategory of $\mathbf{Mod}\text{-}\mathcal{T}^{c}$ consisting of $\Sigma$-injectives. That this functor is full, faithful and dense, follows by Corollary \ref{kraequivpi} together with the fact that any $\Sigma$-pure-injective object of $\mathcal{T}$ is pure-injective.
\end{proof}
Corollary \ref{sigmapropsyanah} is analogous to \cite[Corollary 8.2(i,ii)]{JenLen1989}.
\begin{corollary}\label{sigmapropsyanah} Let $M$ be a $\Sigma$-pure-injective object of $\mathcal{T}$ and let $G\in\mathcal{S}$. 
\begin{enumerate}
\item For any set $\mathtt{I}$, the objects $M^{(\mathtt{I})}$ and $M^{\mathtt{I}}$ are $\Sigma$-pure-injective.
\item If $h\colon L\to M$ is a pure monomorphism in $\mathcal{T}$ then $L$ is $\Sigma$-pure-injective and $h$ is a section.
\end{enumerate}
\end{corollary}
\begin{proof}It is worth noting that we now have the equivalence of (1) and (5) of Theorem \ref{characterisation}. That is, by Lemma \ref{sigmathendcc} and Corollaries \ref{sigmathendccconv} and \ref{sigmaequiv}, an object $M$ is $\Sigma$-pure-injective if and only if every descending chain of pp-definable subgroups of $M$ of each sort stabilises. 

The proof of part (1) is now straightforward, recalling that, by Lemma \ref{ppprops2}, we have that $\varphi(M)^{(\mathtt{I})}\simeq \varphi(M^{(\mathtt{I})})$ and $\varphi(M^{\mathtt{I}})\simeq \varphi(M)^{\mathtt{I}}$ for any pp-formula $\varphi$ of sort $G$. For (2), as above it suffices to recall that $\varphi(L)=\{g\in\mathcal{T}(G,L)\mid hg\in\varphi(M)\}$ for any pp-formula $\varphi$ of sort $G$, by Lemma \ref{ppprops}.
\end{proof}
To prove Lemma \ref{sigmathenCARD} we use Corollary \ref{sigmainjcoprodindpis}, a result of Garcia and Dung.
\begin{corollary}\label{sigmainjcoprodindpis}\emph{\cite[Corollary 1.6]{GarDun1994}} Let $\mathcal{A}$ be a Grothendieck category with a set of finitely generated generators. Any $\Sigma$-injective object of $\mathcal{A}$ is a coproduct of indecomposable objects. 
\end{corollary}
Lemma \ref{sigmathenCARD} is analogous to the implication of (v) by (i) in \cite[Theorem 8.1]{JenLen1989}. The case $\mathtt{I}=\{0\}$ shows that $\Sigma$-pure-injectives decompose into indecomposables.
\begin{lemma}\label{sigmathenCARD} If $M$ is a $\Sigma$-pure-injective object of $\mathcal{T}$ then for any set $\mathtt{I}$ the product $M^{\mathtt{I}}$ is a coproduct of indecomposable $\Sigma$-pure-injectives. 
\end{lemma}
\begin{proof}Recall $\mathbf{Mod}\text{-}\mathcal{T}^{c}$ is a Grothendieck category with a set of finitely generated generators. Let $K=M^{\mathtt{I}}$ so that, by Corollary \ref{sigmapropsyanah}(1), $K$ is $\Sigma$-pure-injective, and so $\mathbf{Y}(K)$ is $\Sigma$-injective by Corollary \ref{sigmaequiv}. By Corollary \ref{sigmainjcoprodindpis} this means $\mathbf{Y}(K)\simeq\coprod_{j\in \mathtt{J}}\mathscr{L}_{j}$ where $\mathscr{L}_{j}$ is an indecomposable object of $\mathbf{Mod}\text{-}\mathcal{T}^{c}$ for each $j$. 

For each $j$, there is a section $u_{j}\colon\mathscr{L}_{j}\to\mathbf{Y}(K)$, which means $\mathscr{L}_{j}$ is injective, and so by Corollary \ref{kraequivpi} we have $\mathscr{L}_{j}\simeq\mathbf{Y}(L_{j})$ for some pure-injective object $L_{j}$. Since $\mathscr{L}_{j}$ is indecomposable, we have that $L_{j}$ is indecomposable. Again, applying Corollary \ref{kraequivpi} gives a section $h_{j}\colon L_{j}\to K$ in $\mathcal{T}$ with $\mathbf{Y}(h_{j})=u_{j}$. Altogether this means $h_{j}$ is a pure monomorphism into a $\Sigma$-pure-injective object, and so $L_{j}$ is $\Sigma$-pure-injective by Corollary \ref{sigmapropsyanah}(2). By Remark \ref{yonedacoprod} we have that $\mathbf{Y}(\coprod_{j}L_{j})\simeq\mathbf{Y}(K)$ is injective, and so $\coprod_{j}L_{j}\simeq K$ by Corollary \ref{kraequivpi}.
\end{proof}
\section{Completing the proof of the characterisation.}\label{5}
Before proving Theorem \ref{characterisation} we note a consequence of the results gathered so far. Recall, from Definitions \ref{sorts2-1} and \ref{compactob}, that the cardinality of the $\mathfrak{L}^{\mathcal{T}}$-structure $\mathsf{M}$ underlying any object $M$ of $\mathcal{T}$ is defined and denoted $\vert\mathsf{M}\vert=\vert\bigsqcup_{G\in\mathcal{S}}\mathcal{T}(G,M)\vert$ where $\mathcal{S}$ is a fixed chosen set of isoclass representatives, one for each class. Corollary \ref{sigmacardy} is analogous to \cite[Corollary 8.2(iii)]{JenLen1989}. 
\begin{corollary}\label{sigmacardy}There exists a cardinal $\kappa$ such that the $\mathfrak{L}^{\mathcal{T}}$-structure underlying any indecomposable pure-injective object of $\mathcal{T}$ has cardinality at most $\kappa$.
\end{corollary}
We delay the proof of Corollary \ref{sigmacardy} until after Corollary \ref{setofindpures}.
\begin{remark}We now note a non-trivial complication in our setting of compactly generated triangulated categories, which is absent in the module-theoretic situation. In the spirit of the results presented so far, it is natural to ask if one may adapt the proof of \cite[Corollary 8.2(iii)]{JenLen1989} to prove Corollary \ref{sigmacardy}. This seems straightforward at first glance, since there are multi-sorted versions of the \textit{downward L\"{o}wenheim}-\textit{Skolem} theorem; see for example \cite[Theorem 37]{DGN2016}. 

Recall that, in the language $\mathfrak{L}_{A}$ of modules over a fixed ring $A$ from Example \ref{onesorted}, any $\mathfrak{L}_{A}$-structure is an object in the category of $A$-modules, and any elementary embedding of $\mathfrak{L}_{A}$-structures is a pure embedding of $A$-modules. The same correspondence between structures need not be true here. Although objects of $\mathbf{Mod}\text{-}\mathcal{T}^{c}$ correspond to structures over $\mathfrak{L}^{\mathcal{T}}$ by Remark \ref{canonlangrem}, the former need not be given by objects of $\mathcal{T}$ as $\mathbf{Y}$ need not be essentially surjective. Fortunately, here we may instead use Corollary \ref{setofindpures}, a remark due to Krause, which also shortens the proof.
\end{remark}
\begin{corollary}\label{setofindpures}\emph{(}See \emph{\cite[Corollary 1.10]{Kra2000}).} There is a set $\mathtt{Sp}$ of isomorphism classes of objects in $\mathcal{T}$ which are pure-injective indecomposable. 
\end{corollary}
\begin{proof}[Proof of Corollary \ref{sigmacardy}]It suffices to let $\kappa=\vert\bigsqcup\mathcal{T}(G,M)\vert$ where $G$ runs through $\mathcal{S}$ and $M$ runs through the set $\mathtt{Sp}$ from Corollary \ref{setofindpures}.
\end{proof}
We now proceed toward proving Theorem \ref{characterisation}. For this we require Corollary \ref{CARDtensigma}, and to this end, in Theorem \ref{REMAK} we recall well-known results about decompositions of coproducts in abelian categories. For consistency we follow \cite{Pre1988}.
\begin{theorem}\label{REMAK}Let $\mathcal{A}$ be an abelian category, and let $\mathrm{L}=\{\mathscr{L}_{j}\mid j\in \mathtt{J}\}$ and $\mathrm{N}=\{\mathscr{N}_{k}\mid k\in \mathtt{K}\}$ be collections of objects in $\mathcal{A}$. The following statements hold.
\begin{enumerate}
\item \emph{\cite[p.82, Theorem 4.A7]{Pre1988}} Suppose the collections $\mathrm{L}$ and $\mathrm{N}$ consist of indecomposable objects with local endomorphism rings. If we have  $\coprod_{j\in \mathtt{J}} \mathscr{L}_{j}\simeq \coprod_{k\in \mathtt{K}} \mathscr{N}_{k}$ then there is a bijection $\sigma\colon \mathtt{J}\to \mathtt{K}$ with $\mathscr{L}_{j}\simeq \mathscr{N}_{\sigma(j)}$ for all $j$.
\item \emph{\cite[p.82, Theorem 4.A11]{Pre1988}} Suppose that $\mathtt{J}=\mathtt{K}$, and that for all $j$, $\mathscr{L}_{j}$ is indecomposable and $\mathscr{L}_{j}$ is the injective hull of $\mathscr{N}_{j}$. Suppose there is a section $\mathscr{P}\to \coprod _{j\in \mathtt{J}}\mathscr{L}_{j}$. Then there is a subset $\mathtt{H}\subseteq \mathtt{J}$ with $\coprod _{j\in \mathtt{J}}\mathscr{L}_{j}\simeq \mathscr{P}\amalg \coprod _{j\in \mathtt{H}}\mathscr{L}_{j}$.
\end{enumerate}
\end{theorem}
To apply Theorem \ref{REMAK} we use the following observation of Garkusha and Prest. 
\begin{lemma}\label{pithenindifflocalendo}\emph{\cite[Lemma 2.2]{GarPre2005}} A pure-injective object $M$ of $\mathcal{T}$ is indecomposable if and only if the endomorphism ring $\mathrm{End}_{\mathcal{T}}(M)$ is local.
\end{lemma}

\begin{corollary}\label{decompcorr}Let $\mathrm{L}=\{L_{j}\mid j\in \mathtt{J}\}$ be a collection of indecomposable objects in $\mathcal{T}$ such that the coproduct $\coprod _{j\in \mathtt{J}}L_{j}$ is pure injective.  If $P$ is a summand of $\coprod _{j\in \mathtt{J}}L_{j}$ then there exists $\mathtt{H}\subseteq \mathtt{J}$ such that $\coprod _{j\in \mathtt{J}}L_{j}\simeq P\amalg \coprod _{j\in \mathtt{H}}L_{j}$. 

If additionally $P\simeq \coprod_{k\in \mathtt{K}} N_{k}$ where each $N_{k}$ is indecomposable, then there is a bijection $\sigma\colon \mathtt{J}\to \mathtt{K}\sqcup \mathtt{H}$ with $L_{j}\simeq N_{\sigma(j)}$ for all $j\in \mathtt{J}$.
\end{corollary}
\begin{proof}Let $\mathscr{L}_{j}=\mathbf{Y}(L_{j})$ for each $j$. Let $L_{\mathtt{T}}=\coprod_{j\in \mathtt{T}}L_{j}$ and $\mathscr{L}_{\mathtt{T}}=\coprod_{j\in \mathtt{T}}\mathscr{L}_{j}$ for any subset $\mathtt{T}\subseteq \mathtt{J}$. Since $\mathbf{Y}$ preserves small coproducts by Remark \ref{yonedacoprod}, and since each $L_{j}$ is a summand of $L_{\mathtt{J}}$, each $\mathscr{L}_{j}$ is a summand of $\mathscr{L}_{\mathtt{J}}$. By Theorem \ref{krachar2} and  Remark \ref{yonedacoprod} $\mathscr{L}_{\mathtt{J}}$ is an injective object. Thus each $\mathscr{L}_{j}$ is injective. Since each $L_{j}$ is indecomposable, each $\mathscr{L}_{j}$ is indecomposable. Similarly, each $N_{k}$ is pure-injective.

Let $\mathscr{P}=\mathbf{Y}(P)$. Since $P$ is a direct summand of $L_{\mathtt{J}}$ by assumption,  $\mathscr{P}$ is a direct summand of $\mathscr{L}_{\mathtt{J}}$. Hence by Theorem \ref{REMAK}(2) there exists a subset $\mathtt{H}\subseteq \mathtt{J}$ such that $\mathscr{L}_{\mathtt{J}}\simeq \mathscr{P}\amalg \mathscr{L}_{\mathtt{H}}$. By Corollary \ref{kraequivpi}, since we assume $L_{\mathtt{J}}$ is pure injective, $\mathscr{L}_{\mathtt{J}}$ is injective, and hence so too are $\mathscr{P}$ and $\mathscr{L}_{\mathtt{H}}$. Again, by Corollary \ref{kraequivpi} this gives $L_{\mathtt{J}}\simeq P\amalg L_{\mathtt{H}}$.

Now suppose also $P\simeq \coprod_{k\in \mathtt{K}} N_{k}$ where each $N_{k}$ is indecomposable and, as above, neccesarilly pure-injective. By Lemma \ref{pithenindifflocalendo} each of the objects $L_{j}$ and $N_{k}$ has a local endomorphism ring. By Corollary \ref{kraequivpi}, for any pure-injective object $Z$ of $\mathcal{T}$ there is a ring isomorphism
\[
\mathrm{End}_{\mathcal{T}}(Z)\simeq\mathrm{End}_{\mathbf{Mod}\text{-}\mathcal{T}^{c}}(\mathbf{Y}(Z)).
\]
Since $\mathscr{L}_{\mathtt{J}}\simeq \mathscr{P}\amalg \mathscr{L}_{\mathtt{H}}$, the second claim follows, as above, by Theorem \ref{REMAK}(1), using again that $\mathbf{Y}$ preserves coproducts by Remark \ref{yonedacoprod}. 
\end{proof}
Lemma \ref{zimmtech3} is analogous to the cited result of Huisgen-Zimmerman.
\begin{lemma}\label{zimmtech3}\emph{\cite[Lemma 5]{Hui1979}} Recall and consider Notation \ref{techynotation}, where $M=\prod_{i\in\mathbb{N}}M_{i}=\coprod_{j\in \mathtt{J}}L_{j}$. Let $X$ be compact and let $\varphi_{1}(M)\supseteq \varphi_{2}(M) \supseteq \dots$ be a descending chain of pp-definable subgroups of sort $X$. Suppose additionally that each object $L_{j}$ is both pure-injective and indecomposable.

Then there exists $r\in\mathbb{N}$ such that, for each collection $\mathrm{N}=\{N_{i}\mid i\in\mathbb{N}\}$ where $N_{i}$ is an indecomposable summand of $M_{i}$, we have $\varphi_{r}(N_{i})=\varphi_{n}(N_{i})$ for all $n\in\mathbb{N}$ with $n\geq r$ and all but finitely many $i\in \mathbb{N}$.
\end{lemma}
\begin{proof}Recall the morphisms $\rho_{n,j}=q_{j,\Psi(n)}\kappa\langle\varphi_{n}\rangle u_{\geq,\Pi(n)}$ defined by the composition
\[
\xymatrix{\prod_{i\geq n}\varphi_{n}(M_{i})\ar[rr]^{u_{\geq,\Pi(n)}} & & \prod_{i\in\mathbb{N}}\varphi_{n}(M_{i})\ar[r]^{\kappa\langle\varphi_{n}\rangle} & \coprod_{j\in \mathtt{J}}\varphi_{n}(L_{j})\ar[rr]^{q_{j,\Psi(n)}} & & \varphi_{n}(L_{j}).
}
\]
Here the isomorphism $\kappa\langle\varphi_{n}\rangle$ is defined in Corollary \ref{zimmtech}, and the maps $u_{\geq,\Pi(n)}$ and $q_{j,\Psi(n)}$ are the canonical inclusion and projection respectively. By Lemma \ref{zimmtech2} there exists $r\in\mathbb{N}$ and a finite subset $\mathtt{J}'$ of $\mathtt{J}$ such that we have the containment $\mathrm{im}(\rho_{r,j})\subseteq \varphi_{n}(L_{j})$ for all $n\in\mathbb{N}$ with $n\geq r$ and all $j\in \mathtt{J}\setminus \mathtt{J}'$. 

Choose arbitrary  $m\in\mathbb{N}$ with $m>\vert \mathtt{J}'\vert$, and choose arbitrary $i_{1},\dots,i_{m}\in\mathbb{N}$ with $i_{p}\geq r$ for each $p$. Let $P=N_{i_{1}}\amalg \dots \amalg N_{i_{m}}$ so that $P$ is a summand of $M=\coprod_{j}L_{j}$. By Corollary \ref{decompcorr}, since $P$ is a coproduct of $m$  indecomposable pure-injective objects, we have $M\simeq P\amalg \coprod_{j\in \mathtt{K}}L_{j}$ and a bijection $\sigma\colon \mathtt{J}\to \mathtt{K}\sqcup\{i_{1},\dots,i_{m}\}$  with $L_{j}\simeq N_{\sigma(j)}$ for all $j\in \mathtt{J}$. 

Since $m>\vert \mathtt{J}'\vert$ there exists $p=1,\dots,m$ such that $\sigma^{-1}(i_{p})\notin \mathtt{J}'$. Let $j=\sigma^{-1}(i_{p})$. It is straightforward to check that post composition with the isomorphism $L_{j}\simeq N_{\sigma(j)}$ defines an isomorphism $\varphi_{n}(L_{j})\simeq \varphi_{n}(N_{\sigma(j)})$ for each $n\in\mathbb{N}$. Since $N_{\sigma(j)}$ is a summand of $M_{\sigma(j)}$ and $i_{p}\geq r$ the group $\varphi_{n}(N_{\sigma(j)})$ embedds into $\prod_{i\geq r}\varphi_{r}(M_{i})$. Considering the image under $\rho_{r,j}$ the result follows by Lemma \ref{zimmtech2} as in the proof of \cite[Lemma 5]{Hui1979}.
\end{proof}
\begin{corollary}\label{CARDtensigma}Let $M$ be a pure-injective object of $\mathcal{T}$ such that $M^{I}$ is a coproduct of indecomposable pure-injectives for any set $I$. Then $M$ is $\Sigma$-pure-injective.
\end{corollary}
\begin{proof}Let $K=M^{\mathbb{N}}$ and $M_{i}=M$ for each $i\in \mathbb{N}$, so that $K=\prod_{i\in\mathbb{N}}M_{i}$. By hypothesis we have $K=\coprod_{j\in \mathtt{J}}L_{j}$ where each $L_{j}$ is an indecomposable pure-injective object of $\mathcal{T}$. By Lemma \ref{zimmtech3} there exists $r\in\mathbb{N}$ such that, for each collection $\mathrm{N}=\{N_{i}\mid i\in\mathbb{N}\}$ where $N_{i}$ is an indecomposable summand of $M_{i}$, we have $\varphi_{r}(N_{i})=\varphi_{n}(N_{i})$ for all $n\in\mathbb{N}$ with $n\geq r$ and all but finitely many $i\in \mathbb{N}$. Fixing $j\in \mathtt{J}$, for the collection given by $N_{i}=L_{j}$ for all $i$, we have  $\varphi_{r}(L_{j})=\varphi_{n}(L_{j})$ for all $n\in\mathbb{N}$, and $\varphi_{n}(K)\simeq\coprod_{j}\varphi_{n}(L_{j})$ for all $n$ by Lemma \ref{ppprops2}.
\end{proof}
Theorem \ref{harfai} goes back to a characterisation due to Faith \cite[Proposition 3]{Fai1966}.
\begin{theorem}\label{harfai}\emph{\cite[Theorem 1]{Har1973} (}see also \emph{\cite[Lemma 1.1]{GarDun1994}).} Let $\mathscr{M}$ be an injective object in a Grothendieck category $\mathcal{A}$ which has a set $\mathtt{G}$ of finitely generated generators. Then the following statements are equivalent.
\begin{enumerate}
\item $\mathscr{M}$ is $\Sigma$-injective.
\item The countable coproduct $\mathscr{M}^{(\mathbb{N})}$ is injective.
\item For each $\mathscr{G}\in\mathtt{G}$, any ascending chain of $\mathscr{M}$-annihilator subobjects of $\mathscr{G}$ eventually stabilises.
\end{enumerate}
\end{theorem}
Finally, we may now complete the proof of our main result.
\begin{proof}[of Theorem \ref{characterisation}] Taking $I=\mathbb{N}$ shows (1) implies (2). That (2) implies (3) follows from Remark \ref{yonedacoprod} and Theorem \ref{harfai}. That (3) implies (1) follows from Remark \ref{locallycoherentremark}, Proposition \ref{fpinjgd} and Lemma \ref{fpinjkra}. The equivalence of (1) and (4) follows from Proposition \ref{sigmapimapsdef2}. That (1) implies (5) follows from Lemma \ref{sigmathendcc}, and the converse follows from Corollaries \ref{sigmathendccconv} and \ref{sigmaequiv}. The equivalence of (1) and (6) follows from Lemma \ref{sigmathenCARD} and Corollary \ref{CARDtensigma}. 
\end{proof}
\section{Applications to endoperfection.}\label{applicationstoendoperfect}
To consider direct applications of Theorem \ref{characterisation}, recall Lemma \ref{endodefinable}, which says that pp-definable subgroups of an object $M$ of a fixed sort $X$ define $\mathrm{End}_{\mathcal{T}}(M)$-submodules of $\mathcal{T}(X,M)$. We begin by introducing the definition of an \emph{endoperfect} object; see Definition \ref{defendoperfect}. This concept is motivated by the idea of a \emph{perfect} module as in Bj{\"o}rk \cite{Bjo1969}, and generalises the \emph{finite endolength} objects defined by Krause \cite{Krause1999}, recalled in Definition \ref{defendofinite}.
\begin{definition}\label{defendoperfect}
An object $M$ of $\mathcal{T}$ is said to be \emph{endoperfect} if for all compact objects $X$ of $\mathcal{T}$ the left $\mathrm{End}_{\mathcal{T}}(M)$-module $\mathcal{T}(X,M)$ has the descending chain condition on cyclic submodules. That is, $M$ is endoperfect provided $\mathcal{T}(X,M)$ is a perfect $\mathrm{End}_{\mathcal{T}}(M)$-module in the sense of \cite[\S 1]{Bjo1969} for each compact object $X$.
\end{definition}
As the terminology suggests, any object with a perfect endomorphism is endoperfect by the famous characterisation of perfect rings due to Bass \cite[Theorem P]{Bas1960}. We now relate this to purity by adapting the notion of \emph{endonoerthian} modules to our categorical setting. This was motivated by work of Huisgen-Zimmermann and Saor{\'i}n \cite{HuiSao2002}, in which conditions for endomorphism rings are related to $\Sigma$-pure-injective modules.  
\begin{definition}\label{defendonoether}
An object $M$ of $\mathcal{T}$ is called \emph{endonoetherian} (respectively, \emph{endoartinian}) if $\mathcal{T}(X,M)$ is noetherian (respectively, artinian) over $\mathrm{End}_{\mathcal{T}}(M)$ for each compact object $X$.
\end{definition}
By using a result of Bj{\"o}rk \cite{Bjo1969} we now see a direct application of Theorem \ref{characterisation}.
\begin{corollary}\label{endonoetherperfectimpliessigmaPI}
Any endoperfect endonoetherian object of $\mathcal{T}$ is $\Sigma$-pure-injective, and hence decomposes into a coproduct of indecomposable objects with local endomorphism rings. 
\end{corollary}
\begin{proof}
Let $X$ be a compact object of $\mathcal{T}$ and $\varphi_{1}(M)\supseteq \varphi_{2}(M)\supseteq \dots$ be a descending chain of pp-definable subgroups of $M$ of sort $X$. By Lemma \ref{endodefinable} this defines a descending chain of $\mathrm{End}_{\mathcal{T}}(M)$-submodules of $\mathcal{T}(X,M)$. Since $M$ is endonoetherian each submodule (of the form $\varphi_{n}(M)$) is finitely generated. Since $M$ is endoperfect the module $\mathcal{T}(X,M)$ has the descending chain condition on cyclic submodules. This means $\mathcal{T}(X,M)$ is perfect (over $\mathrm{End}_{\mathcal{T}}(M)$) in the sense of Bj{\"o}rk \cite[\S 1]{Bjo1969}. By \cite[Theorem 2]{Bjo1969} this means that the  descending chain of finitely generated submodules $\varphi_{1}(M)\supseteq \varphi_{2}(M)\supseteq \dots$ of $\mathcal{T}(X,M)$ must terminate. 

By the equivalence of (1) and (5) in Theorem \ref{characterisation} this means $M$ is $\Sigma$-pure-injective. That $M$ decomposes into indecomposables with local endomorphism rings follows by the equivalence of (1) and (6) in Theorem \ref{characterisation} (setting $\mathtt{I}=\{0\}$).
\end{proof}
Lemma \ref{cyclicendosubmodispp} was motivated by a result due to Crawley-Boevey \cite[\S 4, Lemma]{Cra1992}. 
\begin{lemma}\label{cyclicendosubmodispp}
If $M$ is a pure-injective object of $\mathcal{T}$ and $f\colon X\to M$ is a morphism in $\mathcal{T}$ with $X$ compact, then the cyclic $\mathrm{End}_{\mathcal{T}}(M)$-submodule $\mathrm{End}_{\mathcal{T}}(M)f$ of $\mathcal{T}(X,M)$ is an intersection of pp-definable subgroups of $M$ of sort $X$.
\end{lemma}
\begin{proof}
We follow \cite[\S 1.7]{Cra1992}. Suppose $\mathscr{K}$ is a finitely generated subobject of $\mathrm{ker}(\mathbf{Y}(f))$, so that we have an exact sequence in $\mathbf{Mod}\text{-}\mathcal{T}^{c}$ of the form $\mathbf{Y}(Z)\to \mathscr{K}\to0$ with $Z$ compact. By construction the quotient $\mathbf{Y}(X)/\mathscr{K}$ is finitely presented, so there is a morphism $k\colon Z\to X$ such that $\mathbf{Y}(k)$ is the composition $\mathbf{Y}(Z)\to \mathscr{K}\subseteq \mathbf{Y}(X)$. Similarly it is straightforward to check that any morphism $\mathbf{Y}(X)/\mathscr{K}\to \mathbf{Y}(M)$ is given by a morphism $h\colon X\to M$ in $\mathcal{T}$ such that $hk=0$. 

By \cite[Proposition 3.1]{GarPre2005} this shows that the set of morphisms $\mathbf{Y}(X)/\mathscr{K}\to \mathbf{Y}(M)$ is the set of $\mathbf{Y}(h)$ where $h$ lies in a pp-definable subgroup of $M$ of sort $X$. We now follow \cite[\S 4, Lemma]{Cra1992}. Since $\mathbf{Mod}\text{-}\mathcal{T}^{c}$ is a locally finitely generated Grothendieck category we have that $\mathrm{ker}(\mathbf{Y}(f))$ is the sum $\sum\mathscr{K}$ over the finitely generated subobjects  $\mathscr{K}$ of $\mathrm{ker}(\mathbf{Y}(f))$. 
Hence the set of morphisms $\mathbf{Y}(X)/\mathrm{ker}(\mathbf{Y}(f))\to \mathbf{Y}(M)$ is the intersection over each $\mathscr{K}$ of the set of morphisms $\mathbf{Y}(X)/\mathscr{K}\to \mathbf{Y}(M)$. Altogether this shows  the set of $\mathbf{Y}(X)/\mathrm{ker}(\mathbf{Y}(f))\to \mathbf{Y}(M)$ in $\mathbf{Mod}\text{-}\mathcal{T}^{c}$ is given by the set of $\mathbf{Y}(h)$ where $h$ lies in an intersection of pp-definable subgroup of $M$ of sort $X$. 

The pure-injectivity of $M$ in $\mathcal{T}$ is equivalent, by Corollary \ref{kraequivpi}, to the injectivity of its image $\mathbf{Y}(M)$. Thus, considering the canonical embedding $\mathbf{Y}(X)/\mathrm{ker}(\mathbf{Y}(f))\to \mathbf{Y}(M)$ given by $\mathbf{Y}(f)$, any morphism $\mathbf{Y}(X)/\mathrm{ker}(\mathbf{Y}(f))\to \mathbf{Y}(M)$ factors through an endomorphism of $\mathbf{Y}(M)$. Altogether this shows that the cyclic module $\mathrm{End}_{\mathcal{T}}(M)f$ is the intersection $\bigcap_{\mathscr{K}}\varphi_{\mathscr{K}}(M)$ where each pp-formula $\varphi_{\mathscr{K}}$ is given by $\varphi_{\mathscr{K}}(v_{G})=\exists u_{Z}\colon v_{G}=u_{Z}k$ in the above notation. 
\end{proof}
Hence we may now provide partial a converse to Corollary \ref{endonoetherperfectimpliessigmaPI}.
\begin{corollary}\label{sigmapiimpliesendoperfect}
Any $\Sigma$-pure-injective object $M$ of $\mathcal{T}$ is endoperfect. 
\end{corollary}
\begin{proof} By the equivalence of (1) and (5) in Theorem \ref{characterisation} we have that for each compact object $X$ any descending chain of pp-definable subgroups of $M$ of sort $X$ must stabilise. By Lemma \ref{cyclicendosubmodispp} any cyclic $\mathrm{End}_{\mathcal{T}}(M)$-submodule of $\mathcal{T}(X,M)$ is an interesction, and hence a finite intersection, of such subgroups. Thus any cyclic $\mathrm{End}_{\mathcal{T}}(M)$-submodule of $\mathcal{T}(X,M)$ is a pp-definable subgroup of $M$ of sort $X$. So, by definition, and by the equivalence of (1) and (5) in Theorem \ref{characterisation}, any descending chain of cyclic $\mathrm{End}_{\mathcal{T}}(M)$-submodules  stabilises.
\end{proof}
The author would be interested in finding, if it exists, an endoperfect object in a compactly generated triangulated category which is not $\Sigma$-pure-injective. Of course, by Corollary \ref{endonoetherperfectimpliessigmaPI} such an example would not be endonoetherian.
\begin{definition}\cite[Definition 1.1]{Krause1999}\label{defendofinite}
An object $M$ of $\mathcal{T}$ is called \emph{endofinite} if $\mathcal{T}(X,M)$ has finite length over $\mathrm{End}_{\mathcal{T}}(M)$ for all compact objects $X$. 
\end{definition}
To conclude, we recover a result of Krause. Note endofinite objects are endorartinian, endonoetherian, hence endoperfect and $\Sigma$-pure-injective by Corollary \ref{endonoetherperfectimpliessigmaPI}.
\begin{theorem}\emph{\cite[Theorem 1.2(1)]{Krause1999}}
Every endofinite object of $\mathcal{T}$ decomposes into a coproduct of indecomposable objects with local endomorphism rings. 
\end{theorem}

\bibliography{biblio}
\bibliographystyle{amsplain}
\end{document}